\documentclass{article}
\usepackage{amsmath}
\usepackage{amsthm}
\usepackage{amssymb}
\usepackage{amsfonts}
\usepackage{hyperref}
\usepackage{geometry}
\usepackage{graphicx} 
\usepackage[
backend=biber,
style=alphabetic,
sorting=ynt
]{biblatex}
\usepackage{mathrsfs}
\usepackage{xcolor}

\newtheorem{thm}{Theorem}[section]
\newtheorem{prop}[thm]{Proposition}

\newtheorem{lemma}[thm]{Lemma}

\theoremstyle{definition}
\newtheorem{defn}[thm]{Definition}
\newtheorem{exmpl}[thm]{Example}
\theoremstyle{remark}
\newtheorem*{rmk}{Remark}

\DeclareMathOperator{\Spec}{Spec}

\DeclareMathOperator{\Hilb}{Hilb}

\DeclareMathOperator{\gl}{GL}
\DeclareMathOperator{\sll}{SL}
\DeclareMathOperator{\spp}{Sp}
\DeclareMathOperator{\so}{SO}

\usepackage{graphicx} 

\title{Special Cases of the Shafarevich Conjecture for Complete Intersections in Abelian Varieties}
\author{Frank Lu}
\date{}
\begin{document}

\maketitle
\begin{abstract}
In this paper, we prove the Shafarevich conjecture for certain complete intersections of hypersurfaces in abelian varieties defined over a number field $K$ using the Lawrence-Venkatesh method. The main new inputs we need are computation of certain Euler characteristics of these complete intersections and a big monodromy statement for the variation of Hodge structure arising from the middle cohomology of a family of such complete intersections. Following \cite{ls25}, we prove the latter by relating this monodromy statement to a statement about Tannaka groups, which we then convert into a combinatorial statement.
\end{abstract}
\tableofcontents
\section{Introduction}
Let $A$ be an abelian variety with dimension $n,$ defined over a number field $K.$ Let $S$ be a set of places of $K,$ including all of the archimedean places and the places of bad reduction of $A,$ and let $\mathcal{O}_{K, S}$ be the ring of $S$-integers of $K.$ We can view $A$ as the generic fiber of an abelian scheme $\mathcal{A}$ over $\mathcal{O}_{K, S}.$ 
\par We say that a subvariety $X \subset A$ has \textit{good reduction} at a prime $\mathfrak{p} \not \in S$ if its (Zariski) closure in $\mathcal{A}$ is smooth at $\mathfrak{p}.$ The goal of this paper is to prove the Shafarevich conjecture for certain classes of complete intersections of hypersurfaces in an abelian variety.
\par More specifically, we consider these kinds of subvarieties:
\begin{defn}[See pp. 3-4 of \cite{km24}]
Given a subvariety $X \subset A,$ the \textit{stabilizer subgroup} $\text{Stab}_X(A)$ is the $\overline{K}$-group scheme whose functor of points, for an $\overline{K}$-scheme $T,$ is given by \[\text{Stab}_X(A)(T) = \{a \in A(T)|X_T + a= X_T\}.\]
\par Then, we say that $X \subset A$ is \textit{nondivisible} if this stabilizer subgroup is trivial.
\end{defn}
\par Our main theorem is then the following.
\begin{thm}\label{thm: main theorem}
In the above setting, let $\phi$ be an ample Neron-Severi class, and let $d_1, d_2, \ldots, d_r$ be positive integers. 
\par Then, if $n \geq 10r^4 + 1000,$ then there are only finitely many isomorphism classes of smooth, geometrically connected, and nondivisible subvarieties $X \subset A$ which are the complete intersection of hypersurfaces representing the Neron-Severi classes $d_1\phi, d_2\phi, \ldots, d_r\phi$ with good reduction outside of $S.$ Furthermore, if the $d_i$ are all equal, then the conclusion holds for $n \geq 10r^2 + 1000.$
\end{thm}
We begin by situating this result among other Shafarevich-style results. Classically, Faltings proved the Shafarevich conjecture for curves and abelian varieties in \cite{faltings}. For higher-dimensional varieties in general, the Shafarevich conjecture has been proven for specific families, including for certain hyper-Kahler varieties as in \cite{andreK3}, complete intersections with Hodge level at most $1$ by \cite{JavanLoughranLowHodge}, and flag varieties as in \cite{JLflag}.
\par In the case of subvarieties of an abelian variety, much more is known. There are two main prior results for subvarieties of abelian varieties, occupying complementary dimensional ranges. Lawrence and Sawin in \cite{ls25} prove the Shafarevich conjecture for ample hypersurfaces in abelian varieties (covering $r = 1$ in our setup). Meanwhile, Kramer and Maculan in \cite{km24} prove the Shafarevich conjecture for many subvarieties whose dimension is less than half of that of the abelian variety (covering $n < 2r$).
\par Our main proof strategy uses same circle of ideas as these two papers, utilizing the method first introduced by Lawrence and Venkatesh in \cite{lv18}. In their paper, they provide an alternate proof of the Mordell conjecture which, unlike the original proof of Faltings in \cite{faltings}, doesn't require the use of the Tate and Shafarevich conjectures for abelian varieties. 
\par We briefly summarize the Lawrence-Venkatesh method. Suppose that we are interested in proving the finiteness of $\mathcal{O}_{K, S}$ points for some smooth $K$-variety $T.$ The main idea is to study the variation of Hodge structures arising from the cohomology of a smooth projective family over $T.$ Specifically, one shows two things:
\begin{enumerate}
    \item The integral points land in a small-dimensional locus of the appropriate flag variety parametrizing the possible values for the Hodge filtration. 
    \item Over all $p$-adic points, the Zariski closure of the attained flags in the flag variety has sufficiently large dimension. 
\end{enumerate}
Combining these proves that the integral points are not Zariski dense in $T.$ 
\par By letting $T$ be an appropriate moduli space of varieties, one can utilize the Lawrence-Venkatesh approach to prove various Shafarevich-style conjectures. For instance, Lawrence and Venkatesh also prove in \cite{lv18} that the set of integral points on the moduli space of smooth hypersurfaces with given degree $d$ in $\mathbb{P}^n$ is not Zariski dense.
\par In order to carry out this method, one needs two inputs: the dimensions of the graded pieces of the flags satisfy a mild numerical condition, and the monodromy of this variation of Hodge structures (viewed over the complex numbers) is sufficiently big. This second input is the key theoretical challenge in utilizing the Lawrence-Venkatesh method in higher-dimensional settings.
\par Indeed, in order to prove finiteness using this method, one would need to show that the monodromy of this variation of Hodge structures is big when restricting to positive-dimensional subvarieties of the moduli space $T.$ One would then apply this method to appropriate smaller-dimensional subvarieties of $T$ in order to show that the Zariski closure of the integral points cannot be positive-dimensional. However, the monodromy input in the hypersurfaces result of \cite{lv18} is \cite{beauville}, which only proves big monodromy for the whole moduli space, and not other positive-dimensional subvarieties.
\par To obtain an actual finiteness result for integral points of $T,$ \cite{ls25} introduce a new approach to studying the monodromy, which was done more broadly by Javanpeykar, Kramer, Lehn, and Maculan in \cite{jklm23}. In the situation where one is studying a family $\mathcal{X} \rightarrow T$ of subvarieties of an abelian variety $A,$ this approach allows them to obtain a big monodromy statement for all positive-dimensional subvarieties of $T.$ The idea behind this approach is as follows: one can equip an appropriate quotient of the category of perverse sheaves on $A_{K'}$ with a convolution operation for any algebraically closed field extension $K'$ of $K,$ involving the derived pushforward along the multiplication map $\sigma: A_{K'} \times A_{K'} \rightarrow A_{K'}.$ This gives us a Tannakian category of perverse sheaves, first introduced and studied by \cite{kram2015}. 
\par Now, suppose we are given a family $\mathcal{X} \rightarrow T,$ and $K'$ is the field corresponding to a geometric generic point $\overline{\eta}$ of $T.$ We consider the Tannakian subcategory generated by intersection complex $\delta_{\mathcal{X}_{\overline{\eta}}}$ of $\mathcal{X}_{\overline{\eta}} \subset A_{K'}.$ Then, unless this family is essentially constant, in order to prove big monodromy for $\mathcal{X} \rightarrow T,$ it suffices (see \cite[Theorem 4.10]{jklm23} for a precise statement) to show that the Tannaka group of this Tannakian category is sufficiently large in the following sense.
\begin{defn}[See pp.3,5 of \cite{jklm23}]
\par We say that a subvariety $\mathcal{X} \subset A_S$ is \textit{symmetric up to translation} if there exists an element $a \in A(S)$ such that $\mathcal{X} = -\mathcal{X} + a.$
\par Now, suppose that $\mathcal{C}$ be the Tannakian category described above, with the faithful exact functor $\omega: \mathcal{C} \rightarrow \text{Vect}_F$ and Tannaka group $G.$ Let $V = \omega(\delta_{\mathcal{X}_{\overline{\eta}}}),$ and let $G^*_{X, \omega}$ be the derived subgroup of the identity component of $G.$ We say that the group $G^*_{X, \omega}$ is \textit{big} if
\begin{enumerate}
    \item $G^*_{X, \omega} = \sll(V)$ if $\mathcal{X}$ is not symmetric up to translation,
    \item $G^*_{X, \omega} = \spp(V, \theta)$ if $\mathcal{X}$ is symmetric up to translation and $X$ is odd-dimensional, 
    \item $G^*_{X, \omega} = \so(V, \theta)$ if $\mathcal{X}$ is symmetric up to translation and $X$ is even-dimensional.
\end{enumerate}
Here, $\theta$ is a bilinear form arising from Poincare duality; see \cite[p.2]{jklm23} for more details.
\end{defn}
\par Thus, to prove Theorem \ref{thm: main theorem}, it suffices to prove the following.
\begin{thm}\label{thm: big monodromy main theorem}
Suppose that $T$ is a smooth irreducible variety over an algebraically closed field $k,$ with generic point $\eta.$ Let $A$ be an abelian variety over $k,$ and $\mathcal{X} \subset A_T$ a closed subscheme which is smooth over $T.$ 
\par Suppose that $\mathcal{X}_{\overline{\eta}},$ the fiber at the geometric point over $\eta,$ is the complete intersection of hypersurfaces representing the Neron-Severi classes $d_1\phi, d_2\phi, \ldots, d_r\phi,$ where $\phi$ is an ample Neron-Severi class of $A$ and $d_1, d_2, \ldots, d_r$ are positive integers. Suppose furthermore that this fiber is smooth, geometrically connected, has ample normal bundle, and is nondivisible. 
\par Then, for $n \geq 10r^4 + 1000,$ $G^*_{X, \omega}$ is big. Furthermore, if the $d_i$ are all equal, then the above statement holds for $n \geq 10r^2 + 1000.$
\end{thm} 
\par To prove that this group $G^*_{X, \omega},$ we can use the fact that, in the cases considered in our paper, there are very few possibilities for what the pairs of $(G_{X, \omega}^*, \omega(\delta_{\mathcal{X}_{\overline{\eta}}}))$ can be; see \cite[Theorem 6.1]{jklm23}, for instance. Therefore, to prove the Tannaka group is big, it suffices to rule out a few other cases, as enumerated in \cite[Theorem 5.1.5]{Green_2013}, for instance.
\par Currently, there are two main strategies to rule out these cases. In \cite{jklm23}, they employ geometric arguments to rule out the other cases, utilizing the fact that they are interested in high-codimensional subvarieties. Meanwhile, in \cite{ls25}, they rule out all but one of the cases by considering the dimension of $\omega(\delta_{\mathcal{X}_{\overline{\eta}}}).$ The final case they need to rule out is the most intricate, and one can reduce ruling out this case to showing that an intricate system of equations has no solution.
\par We adopt this latter strategy, where the equations we obtain for this final case take the same form. The corresponding statement we show is the following.
\begin{thm}\label{thm: big monodromy statement intro version}
Suppose that $X$ is the complete intersection of $r$ hypersurfaces in $A,$ all representing positive integer multiples of the same ample Neron-Severi class $H,$ namely by $d_1H, d_2H, \ldots, d_rH.$ Then, if $\dim A = n \geq 10r^4 + 1000,$ there is no solution in the functions $m_H(w): \mathbb{Z} \rightarrow \mathbb{Z},$ positive integers $k$ between $2$ and $m - 1$ (where $m = \sum\limits_i m_H(i)$), and integer $s,$ to the following system of equations, indexed by $q = 0, 1, 2, \ldots, n-r:$
\[\sum\limits_{\substack{m_S: \mathbb{Z} \rightarrow \mathbb{Z} \\ 0 \leq m_S(i) \leq m_H(i) \\ \sum m_S(i) = k \\ \sum im_S(i) = s + q}} \prod\limits \binom{m_H(i)}{m_S(i)} = |\chi(X, \Omega^q_{X/K})|.\]
\par Furthermore, if the $d_i$ are all equal, there is no solution for $n \geq 10r^2 + 1000.$
\end{thm}
\par We use part of the strategy of \cite[Appendix B]{ls25} to show that, for sufficiently large $n,$ no solutions can exist. The main idea is as follows: assuming a solution for $m_H$ existed, one can impose a restriction on $m_H$ using the fact that there are only equations for $q = 0$ to $n-r.$ This also allows one to describe the indexing functions $m_S$ corresponding to small values of $q.$ From there, after dividing all the equations by the $q = 0$ equation, one can then interpret the equations as being single-variable polynomials in a suitable expression, where the $q$th equation is degree $d_q \leq q.$ Thus, the existence of a solution to the above system of equations implies that for each $q,$ $\left|\frac{\chi(X, \Omega^q_{X/K})}{\chi(X, \mathcal{O}_X)}\right|$ is either roughly as large as $\left|\frac{\chi(X, \Omega^1_{X/K})}{\chi(X, \mathcal{O}_X)}\right|^q,$ or the sum of secondary terms.
\par At the same time, Lawrence and Sawin interpret their Euler characteristics $\chi(X, \Omega^q_{X/K})$ in terms of the Eulerian numbers, which count the number of permutations with given numbers of descents. This combinatorial interpretation allows them to obtain asymptotics for $\frac{\chi(X, \Omega^q_{X/K})}{\chi(X, \mathcal{O}_X)}$ as $n$ grows large. Combining these two asymptotics reveals the contradiction, as $\frac{\chi(X, \Omega^q_{X/K})}{\chi(X, \mathcal{O}_X)}$ grows faster than the secondary terms but slower than $\left|\frac{\chi(X, \Omega^1_{X/K})}{\chi(X, \mathcal{O}_X)}\right|^q.$
\par To carry out this strategy in our setting, we first compute the constants on the right-hand side to obtain the analogous asymptotics for $\chi(X, \Omega^q_{X/K}).$ We achieve this by computing the more general Euler characteristics $\chi(X, \Omega^q_{X/K}(a_1H_1 + \cdots + a_mH_m)),$ where $X$ the complete intersection of any ample hypersurfaces $H_1, H_2, \ldots, H_m,$ and $a_1, a_2, \ldots, a_m \in \mathbb{Z}.$ These Euler characteristics are computed inductively via the Koszul exact sequence arising from the complete intersection hypothesis and the conormal exact sequence. We then express these in terms of the Eulerian numbers and a generalization of them that we introduce in this paper. This combinatorial interpretation provides the needed asymptotics for $\chi(X, \Omega^q_{X/K}),$ which we use to carry out the strategy for all $r$ at once in the more restricted setting of the above theorem. 
\par Since the big Tannaka group result is known for complete intersections of hypersurfaces when $n < 2r$ by \cite{jklm23}, if one could replace the condition on $n$ with $n \geq 2r,$ one would rule out one the main cases needed to prove Theorem \ref{thm: big monodromy main theorem} for all complete intersections representing multiples of the same ample Neron-Severi class, which is enough to prove the theorem for $n > 2r.$ We remark, however, that the inequalities $n \geq 10r^4 + 1000$ and $n \geq 10r^2 + 1000,$ which arise from making ``sufficiently large $n$" explicit, are not sharp. In particular, by a more careful application of the above argument, one could improve the constants or modestly improve the exponent. However, reaching $n \geq 2r$ would likely require some new ideas.
\par Our paper is organized as follows. In Section \ref{section: eulerian}, we introduce a generalization of the Eulerian numbers, which appear in the formulas for the Euler characteristics we compute in Section \ref{section: cohomology computation}. In particular, for our smooth complete intersection $X$ in an abelian variety $A$ of dimension $n,$ we are interested in computing $\chi(X, \Omega^p_X)$ for $0 \leq p \leq \dim X.$ We additionally show that these Euler characteristics satisfy a certain mild inequality which, by \cite{km24}, is enough to prove the Shafarevich conjecture for complete intersections where $\dim X < \dim A/2.$
\par The following section, Section \ref{section: reduce to combinatorics}, goes through the reduction of our problem to the equations alluded to above, using \cite[Theorem D]{km24} to prove Theorem \ref{thm: main theorem}, before reducing Theorem \ref{thm: big monodromy main theorem} to the combinatorial statement. Finally, in Section \ref{section: big monodromy asymptotics}, we prove the combinatorial statement with the asymptotic argument, relying on various inequalities involving generalized Eulerian numbers that we relegate to the appendix. The constraints relating $n$ and $r$ appearing in Theorem \ref{thm: main theorem} emerge here, as the thrust of the argument relies on having $n$ sufficiently large compared to $r.$ 
\par Before proceeding with the argument, we introduce some notation that will be useful for us. First, we will often want to consider collections of functions $\epsilon: \{1, 2, \ldots, r\} \rightarrow \{0, 1, \ldots, k\}.$ We will call the collection of such functions $F(r, k),$ and let $|\epsilon|$ be the sum of $\epsilon(i)$ over all $i.$
\par Furthermore, it will be useful to us to consider a slightly larger family of subvarieties than just complete intersections. We are interested in varieties that are ``like" complete intersections in the following sense.
\begin{defn}
Given two subvarieties $X, Y \subset A,$ we say that $X$ is \textit{numerically like $Y$} if $\dim X = \dim Y,$ and $\chi(X, \Omega^i_X) = \chi(Y, \Omega^i_Y)$ for $0 \leq i \leq \dim X.$
\end{defn}
For the rest of the paper, fix a number field $K,$ an abelian variety $A$ defined over $K$ with dimension $n,$ and a set of places $S$ which include all archimedean places and places where $A$ has bad reduction.
\subsection*{Acknowledgements}
The author would like to thank Mark Kisin for introducing him to the circle of ideas around this paper, in particular pointing him to \cite{lv18} and \cite{km24}, as well as for enlightening conversations around the broader set of ideas in this area and for feedback on drafts of this paper.
\par In addition, the author used OpenAI's o4-mini model to aid the revision of the introduction of this paper.
\section{Eulerian Numbers and Generalizations}\label{section: eulerian}
In this section, we will review some facts about Eulerian numbers, before describing a generalization which will be useful later on in the paper. For more information about the Eulerian numbers, see \cite[\S 1.3 - 1.5]{eulerian}, which we will draw from for the exposition that we give here.
\subsection{A Review of Eulerian Numbers}
\par We start with the Eulerian numbers, recalling their definition and collecting some useful properties.
\begin{defn}
The Eulerian number $E(n, k),$ for integers $k \geq 0$ and $n \geq 1,$ is equal to the number of permutations $w$ of the set $\{1, 2, \ldots, n\}$ such that there are $k$ descents. That is, there are $k$ elements $i \in \{1, 2, \ldots, n\}$ for which $w(i) > w(i+1).$ 
\end{defn}
\par We first observe the basic property that, if $n \geq 1,$ $\sum\limits_{k=0}^{n-1} E(n, k) = n!;$ every permutation of $n$ elements has between $0$ and $n$ descents. In addition, we have the recurrence relation (see \cite[Theorem 1.3]{eulerian}), given by $E(n, k) = (n-k)E(n-1, k-1) + (k+1)E(n-1, k).$
\par We can also write an explicit formula for the Eulerian numbers as an alternating summation of binomial coefficients.
\begin{prop}[Corollary 1.3 from \cite{eulerian}]\label{prop: alternating summation formula for Eulerian numbers}
We have that for $n \geq 1$ that \[E(n, k) = \sum\limits_{i=0}^k (-1)^i (k + 1 - i)^n \binom{n+1}{i}.\]
\end{prop}
\begin{rmk}
Later, we will encounter situations where we want to write $E(n, k)$ for values of $n, k$ outside of those that we had initially defined. For us, we will say that $E(n, k)$ is zero whenever $n \leq 0,$ or when $k$ doesn't lie between $0$ to $n-1.$ This will allow us to ensure consistency with the algebraic formulas we will handle in the next section.
\end{rmk}
Finally, we note that the sequence $E(n, k),$ for $k = 0, 1, 2, \ldots, n-1,$ is log-concave. This follows from \cite[Exercises 4.6, 4.8]{eulerian}.
\subsection{Generalizing the Eulerian Numbers}\label{subsection: generalized eulerian}
We now introduce a generalization the Eulerian numbers and discuss the properties we will need. These will appear naturally in the context of the Euler characteristics $\chi(X, \Omega^i_{X/K})$ for a complete intersection of hypersurfaces $X.$ 
\begin{defn}
The \textit{$r$th order Eulerian number} for $k \geq 0$ and $n \geq 1$ is given by the number of permutations $w$ of the set $\{1, 2, \ldots, n\},$ along with a subset $S \subset \{1, 2, \ldots, n-1\}$ of size $r-1,$ such that there are $k$ descents $i$ that aren't in $S.$
\end{defn}
In other words, we are counting the number of collections of ``ignored consecutive indices" and permutations with a given number of descents that aren't among the beginnings of these ignored consecutive indices.
\begin{exmpl}\label{exmpl: example of generalized Eulerian}
Consider $n = 4.$ We have that $E_2(4, 0) = 14, E_2(4, 1) = 20, E_2(4, 2) = 14:$ note that we have $3$ choices of this subset, being either $\{1\}, \{2\}, \{3\}.$ For instance, to determine $E_2(4, 0),$ we can start with $S = \{1\}.$ Then, the permutations $\sigma$ with no descents outside of $S$ are those with $\sigma(2) < \sigma(3) < \sigma(4),$ so the permutation is uniquely determined by $\sigma(1)$ (as then we arrange the other $3$ elements in order). This contributes $4;$ similarly with $S = \{2\}$ we choose two elements to be the first two elements given in the permutation (which we put in increasing order), yielding $6$ such permutations, and finally $S = \{3\}$ has $4$ such permutations.   
\end{exmpl}
From the logic of the example above, we can check that \[E_r(n, k) = \sum\limits_{\substack{\epsilon \in F(r, n - r) \\ |\epsilon| = n - r}}\sum\limits_{\substack{t \in F(r, k) \\ |t| = k}} \binom{n}{\epsilon(1) + 1, \epsilon(2) + 1, \ldots, \epsilon(r) + 1}\prod\limits_{i=1}^r E(\epsilon(i) + 1, t(i)).\] Indeed, we first choose our subset $S \in \{1, 2, \ldots, n-1\};$ suppose this is given by $x_1 < x_2 < \cdots < x_{r-1}.$ Then, we consider the function sending $\epsilon(i)$ to $x_i - x_{i-1} - 1$ (saying that $x_0 = 0$ and $x_r = n$). For each subset, we have $\binom{n}{\epsilon(1), \epsilon(2), \ldots, \epsilon(r)}$ ways to assign numbers to the buckets $\{1, \ldots, x_1
\}, \{x_1 + 1, x_1 + 2, \ldots, x_2\}, \ldots, \{x_{r-1} + 1, \ldots, n\},$ and then the product of these Eulerian numbers to figure out how to permute the numbers within each bucket. Conversely, to each such function, we can associate the subset $S$ given by $\{\epsilon(1) + 1, \epsilon(1) + \epsilon(2) + 2, \ldots, \sum\limits_{i=1}^{r-1} \epsilon(i) + r - 1\}.$
\par Since we know that $E(0, k) = 0$ for all integers $k,$ we thus obtain the following.
\begin{lemma}\label{lem: generalized Eulerian in terms of Eulerian}
For any nonnegative integers $n, k,$ we have that \[E_r(n, k) = \sum\limits_{\substack{\epsilon \in F(r, n)\\ |\epsilon| = n}} \sum\limits_{\substack{t \in F(r, k)\\ |t| = k}} \binom{n}{\epsilon(1), \epsilon(2), \ldots, \epsilon(r)} \prod\limits_{i=1}^r E(\epsilon(i), t(i)).\]
\end{lemma}
\par From our definition of generalized Eulerian numbers, we can construct the following recurrence relation relating the larger order Eulerian numbers to the smaller ones (for $r \geq 2$).
\[(r-1)E_r(n, j) = (j+1)E_{r-1}(n, j+1) + (n+1 -j - r)E_{r-1}(n, j).\]
\par Indeed, both of these count the number tuples of a permutation, a choice of $(r-1)$ ignored consecutive pairs, and a specific choice of one of these consecutive pairs. On the right-hand side, either the consecutive pair we chose had a descent (which yields the first term), or it doesn't (yielding the second term). 
\par We can use this modified recurrence relation to prove log concavity of $E_r(n, k),$ for $k = 0, 1, 2, \ldots, n - r,$ by induction on $r.$ The case for $r = 1$ follows by \cite[Problems 4.6, 4.8]{eulerian}, as we mentioned above. Meanwhile, given log concavity for $r-1,$ we want to show that $E_r(n, q)^2 - E_r(n, q-1)E_r(n, q+1) \geq 0$ for all $n, q.$ Substituting in the recurrence formula above (and multiplying by $(r-1)^2$), we obtain \begin{align*}& (j+1)^2E_{r-1}(n, j+1)^2 + 2(j+1)(n+1-j-r)E_{r-1}(n, j+1)E_{r-1}(n, j) + (n+1 -j - r)^2E_{r-1}(n, j)^2 \\ & - (jE_{r-1}(n, j) + (n+2 -j - r)E_{r-1}(n, j - 1))((j+2)E_{r-1}(n, j+2) + (n -j - r)E_{r-1}(n, j+1))  \\ &= E_{r-1}(n, j+1)^2 + j(j+2)\left(E_r(n, j + 1)^2 - E_r(n, j)E_r(n, j+2)\right) - 2E_{r-1}(n, j)E_{r-1}(n, j+1) \\ & + (j+2)(n + 2 - j - r) \left(E_{r-1}(n, j)E_{r-1}(n, j+1) - E_{r-1}(n, j-1)E_{r-1}(n, j + 2)\right) \\ & + E_{r-1}(n, j)^2 + (n + 2 - j - r)(n - j- r)\left(E_{r-1}(n, j)^2 - E_{r-1}(n, j-1)E_{r-1}(n, j+1) \right).\end{align*} But this is at least $0,$ using the log concavity for $E_{r-1}(n, \bullet)$ from our inductive hypothesis (with the remaining terms being $(E_{r-1}(n, j) - E_{r-1}(n, j-1))^2 \geq 0$).
\par Note that $E_r(n, j)$ is also symmetric around $\frac{n-r}{2}$ by definition, so in particular $E_r(n, j)$ is increasing between $j = 0$ and $j = \lfloor \frac{n-r}{2} \rfloor.$
\par Using the properties we've collected above, we now generalize \cite[Lemma C.1]{ls25}, which we will heavily utilize later in the paper.
\begin{lemma}\label{lem: Eulerian asymptotic}
For all $n \geq 1, r \geq 1,$ and $k \geq 0,$ we have that $$\binom{k + r - 1}{r-1} \left((k + r)^n - (n+1)(k + r - 1)^n\right) \leq E_r(n, k) \leq \binom{k+r-1}{r-1} (k+r)^n.$$
\end{lemma}
\begin{proof}
The argument of \cite[Lemma C.1]{ls25} can be easily generalized here. Consider the following construction: assign each of the numbers $1, 2, \ldots, n$ a label between $1$ and $k + r,$ and pick $r-1$ of the labels between $1$ and $k + r-1.$ Then, we consider the permutation where we arrange all the numbers with the label $1$ first in increasing order, then all the numbers with label $2,$ and so forth.
\par Note that each pair of a subset $S$ of $\{1, 2, \ldots, n-1\}$ of size $r-1$ and a permutation with $k$ descents outside of $S$ can be uniquely constructed in this way. To do this, label all the numbers at or before the first descent or element of $S$ with $1,$ then label all the numbers before the next descent or element of $S$ with a $2,$ and so forth. Our choice of a subset of $r-1$ of these labels $1, 2, \ldots, k+r-1$ are those labels $i$ such that the last element with label $i$ is an element of $S.$ This establishes the upper bound.
\par For the lower bound, suppose the data of labels of numbers $1, 2, \ldots, n,$ along with the choices of labels, doesn't correspond to a pair of a subset of $r$ elements and a permutation with $k$ descents outside of these elements. We want to rule out these cases. Then, either there are less than $k + r$ labels, or there is some label $i$ for which the largest element with label $i$ is smaller than the smallest element with label $i + 1.$ Otherwise, we have fewer than $k + r - 1$ numbers in $\{1, 2, \ldots, n-1\}$ which are either descents or elements in $S$ by our above construction. 
\par We can then construct such data as follows. Label the numbers $\{1, 2, \ldots, n\}$ with labels between $1$ and $k + r - 1,$ then pick $r-1$ of these labels. We then pick some number $j$ between $0$ and $n,$ and then increase the labels of all the numbers that either have larger labels, or the same label but are larger numbers. We check that this produces all of the data in the cases we wanted to rule out. 
\par For the first, if label $i$ is skipped, and number $j$ has label $a_j,$ in our new construction we assign labels where $j$ has label $a_j$ if $a_j < i$ and $a_j - 1$ if $a_j > i.$ For the second, if the largest element with label $i,$ which we denote $k,$ is smaller than the smallest element with label $i + 1,$ then in our new construction we assign label $a_j$ to $j$ if $a_j < i$ or $a_j = i$ and $j \leq k,$ or $a_j - 1$ otherwise. Note furthermore that these constructions are distinct.
\par This then gives us $\binom{k+r-1}{r-1} (n+1)(k+r-1)^n$ ways to build this new data. This is then an upper bound for the number of cases we want to rule out and thus a lower bound for $E_r(n, k).$
\end{proof}
We will want an analogous bound to \cite[Lemma C.3]{ls25}, to express both sides in terms of $(k+r)^n.$ We achieve this by using the above lemma.
\begin{lemma}\label{lem: more useable asymptotic}
Suppose that we are in either one of these situations:
\begin{enumerate}
    \item $n \geq 3r^2 + 100,$ $r \geq 2,$ and $k \leq 4r.$
    \item $k \leq \frac{n}{\ln(n+1) + 1} - r.$
\end{enumerate} Then, $$E_r(n, k) \geq 0.6 \binom{k+r-1}{r-1}(k+r)^n.$$ 
\end{lemma}
\begin{proof}
We look at the error term for the lower bound in Lemma \ref{lem: Eulerian asymptotic}; we will show that \[(n+1) \left(1 - \frac{1}{k+r}\right)^n \leq \frac{1}{e} < 0.4.\] We observe that $(1 - \frac{1}{x})^x$ increases to $1/e$ as $x$ goes to positive infinity. Now, for the second case, we find that \[(n+1) \left(1 - \frac{1}{k+r}\right)^n = (n+1)\left(\left(1 - \frac{1}{k+r}\right)^{(k+r)}\right)^{n/(k+r)} \leq (n+1) e^{-n/(k+r)}.\]
\par For this to be less than $1/e,$ we want $\ln(n+1) - \frac{n}{(k+r)} \leq -1.$ However, observe that this holds if and only if $\ln (n+1) + 1 \leq \frac{n}{(k+r)},$ or that $k \leq \frac{n}{\ln(n+1) + 1} - r,$ which is what we're assuming in the second case.
\par In the first case, we can bound the left-hand side by \[(n+1) \left(1 - \frac{1}{5r}\right)^n \leq (n+1) e^{-n/5r} = e^{\ln(n+1) - n/5r}.\]
\par Now, observe that the exponent is decreasing, since the derivative is $\frac{1}{n+1} - \frac{1}{5r} < 0$ for $n \geq 5r - 1$ and $r \geq 1.$ First, we prove that $\ln (3r^2 + 21) - \frac{3r^2 + 20}{5r} < -1$ for large enough $r.$ Again, we take the derivative, where we find that this is $\frac{6r}{3r^2 + 21} - \frac{3}{5} + \frac{4}{r^2},$ which is negative when $r \geq 4.$ But we can see that the inequality above is then satisfied for $r \geq 11,$ plugging in $r = 11.$ Meanwhile, we can manually check this for $2 \leq r \leq 10$ that the above inequality holds for when $3r^2 + 20$ is replaced with $3r^2 + 100,$ which is what we wanted to show.
\end{proof}
\section{Cohomology of Complete Intersections}\label{section: cohomology computation}
We now begin computing the Euler characteristics $\chi(X, \Omega^q_X)$ for $X$ a complete intersection of $r$ hypersurfaces, which we will need in order to prove our desired finiteness statement. In addition, we will verify that these Euler characteristics satisfy the numerical condition in \cite[Theorem D]{km24}, which is one of the conditions we need in order to apply the output of the Lawrence-Venkatesh method. Already this computation will let us utilize \cite[Theorem D]{km24} to prove that the Shafarevich conjecture holds for complete intersections of dimension less than $n/2.$
\subsection{Preliminary Exact Sequences}
\par To begin, let $X$ be the smooth complete intersection of hypersurfaces $H_1, H_2, \ldots, H_r$ of $A$ representing ample Neron-Severi classes. We will start by collecting some exact sequences that we will utilize throughout this section.
\begin{lemma}\label{lem: global Koszul}
Suppose $X$ is the complete intersection of hypersurfaces $H_1, H_2, \dots, H_r$ in a smooth variety $A.$ Then, if $\mathcal{I}_X$ is the ideal sheaf cutting out $X,$ then we have the exact sequence \[0 \rightarrow \bigoplus\limits_{\substack{S \in \{1, 2, \ldots, r\}\\|S| = r}} \mathcal{O}_A(-\sum\limits_{i \in S} H_i) \rightarrow \bigoplus\limits_{\substack{S \in \{1, 2, \ldots, r\}\\|S| = r-1}} \mathcal{O}_A(-\sum\limits_{i \in S} H_i) \rightarrow \cdots \bigoplus\limits_{\substack{S \in \{1, 2, \ldots, r\}\\|S| = 1}} \mathcal{O}_A(-\sum\limits_{i \in S} H_i) \rightarrow \mathcal{I}_X \rightarrow 0,\] where we interpret the hypersurfaces as divisors.
\end{lemma}
\begin{proof}
To prove this, we first recall that $A$ is smooth, ergo factorial, and therefore all Weil divisors are locally principal (see \cite[Lemma 12.1.7]{vakil} and \cite[0AG0]{stacks-project}, for instance). In particular, we have an affine open cover $U_i = \Spec R_i$ of $A$ such that on each open subset $U_i$ in the cover, the hypersurfaces $H_1, H_2, \ldots, H_r$ are cut out by $f_{1i}, f_{2i}, \ldots, f_{ri} \in R_i.$ Since these are complete intersections, we know that $f_{ji}$ is a nonzerodivisor in $R_i/\langle f_{1i}, f_{2i}, \ldots, f_{(j-1)i} \rangle$ for each $j = 1, 2, \ldots, r.$
\par Now, we define the morphisms in the above exact sequence as follows. Given $S \subset \{1, 2, \ldots, r\},$ and $S' \subset S$ arising from deleting the $j$th smallest element of $S,$ we consider the morphism $\mathcal{O}_A(\sum\limits_{i\in S} -H_i) \rightarrow \mathcal{O}_A(\sum\limits_{i\in S'} -H_i),$ arising from $(-1)^{j+1}$ times the inclusion. Similarly, we have the map $\mathcal{O}_A(-H_i) \rightarrow \mathcal{I}_X$ by the inclusion morphism.
\par To prove exactness of the above sequence, it suffices to check exactness on sections of the affine opens $\Spec R_i.$ But our above sequence then looks like the following sequence of $R_i$-modules:
\[0 \rightarrow \bigoplus\limits_{\substack{S \in \{1, 2, \ldots, r\}\\|S| = r}} \left\langle \prod\limits_{j \in S }f_{ji} \right\rangle \rightarrow \bigoplus\limits_{\substack{S \in \{1, 2, \ldots, r\}\\|S| = r-1}} \left\langle \prod\limits_{j \in S }f_{ji} \right\rangle  \rightarrow \cdots \bigoplus\limits_{\substack{S \in \{1, 2, \ldots, r\}\\|S| = 1}} \left\langle \prod\limits_{j \in S }f_{ji} \right\rangle \rightarrow \langle f_1, f_2, \ldots, f_r \rangle \rightarrow 0,\] with all the morphisms given by the appropriate inclusions. This sequence being exact is equivalent to this being exact:
\[0 \rightarrow \bigoplus\limits_{\substack{S \in \{1, 2, \ldots, r\}\\|S| = r}} \left\langle \prod\limits_{j \in S }f_{ji} \right\rangle \rightarrow \bigoplus\limits_{\substack{S \in \{1, 2, \ldots, r\}\\|S| = r-1}} \left\langle \prod\limits_{j \in S }f_{ji} \right\rangle  \rightarrow \cdots \bigoplus\limits_{\substack{S \in \{1, 2, \ldots, r\}\\|S| = 1}} \left\langle \prod\limits_{j \in S }f_{ji} \right\rangle \rightarrow R_i \rightarrow R_i/\langle f_1, f_2, \ldots, f_r \rangle \rightarrow 0.\]
\par This, however, is exact because $f_1, f_2, \ldots, f_r$ is a regular sequence (since we are starting with a complete intersection), where we can employ \cite[Tag 062F]{stacks-project}, which proves the claim.
\end{proof}
In order to compute $\chi(X, \Omega_X^q),$ we will also want to make explicit the conormal short exact sequence. Let $i: X \rightarrow A$ be the closed immersion.
\begin{lemma}\label{lem: complete intersection conormal}
We have the isomorphism $i^*\mathcal{I}_X \simeq \bigoplus\limits_{j=1}^r i^*\mathcal{O}(-H_j),$ with the morphism arising from the last morphism in Lemma \ref{lem: global Koszul}. 
\end{lemma}
\begin{proof}
We obtain this isomorphism from \cite[Proposition 21.2.16]{vakil}, since affine-locally around $X$ (where the hypersurfaces are principal) we can check that the generators of $\mathcal{I}_X/\mathcal{I}_X^2$ are given by the generators of $\mathcal{I}_X,$ which are given by $i^*\mathcal{O}_A(-H_i).$ 
\end{proof}
\par In particular, this means that the conormal exact sequence for $X$ looks like the following:
\[0 \rightarrow \bigoplus\limits_{j=1}^r i^*\mathcal{O}(-H_j) \rightarrow i^*\Omega^1_{A/K} \rightarrow \Omega_{X/K}^1 \rightarrow 0.\]
\par Finally, we recall for an abelian variety that the cotangent bundle is actually trivial.
\begin{lemma}[Corollary 3 from \S 4.2 of \cite{BLR}]
$\Omega^i_{A/K}$ is a trivial vector bundle of rank $\binom{n}{i}.$
\end{lemma}
\subsection{Computing the Euler Characteristics}
We now aim to compute $\chi(\Omega^q_{X/K})$ as a function of $q.$ To do this, we will compute a more general expression, namely $$\chi(\Omega^q_{X/K} \otimes i^*\mathcal{O}_A(d_1H_1 + d_2H_2 + \cdots + d_rH_r + mH)),$$ where $H$ is any given divisor, as a polynomial of the variables $d_1, d_2, \ldots, d_r, m \in \mathbb{Z},$ as we can utilize induction to relate these polynomials together. For almost all of our computations, we will only need to assume $m = 0;$ however for $q = 0$ we will want to take $m$ to be any integer. For ease of notation, we will write \[\Omega^q_{X/K} \otimes i^*\mathcal{O}_A(d_1H_1 + d_2H_2 + \cdots + d_rH_r + mH) = \Omega^q_{X/K}(d_1H_1 + d_2H_2 + \cdots + d_rH_r + mH).\] Let $P_q(d_1, d_2, \ldots, d_r, m)$ be this Euler characteristic; we will let this polynomial be identically zero for $q < 0.$
\par To start, note that taking the wedge product of the conormal exact sequence of vector bundles, we find that $i^*\Omega^q_{A/K} = \mathcal{O}_X^{\binom{n}{q}}$ has a filtration, with graded pieces arising from $\Omega^p_{X/K} \otimes \bigwedge^{q-p} (\bigoplus\limits_{j=1}^r i^*\mathcal{O}(-H_j)).$ See \cite[Exercise 14.2.N]{vakil}, for instance. As tensoring with line bundles is exact, and Euler characteristics add in exact sequences, we thus obtain the recurrence relation \[\binom{n}{q} P_0(d_1, d_2, \ldots, d_r, m) = \sum\limits_{s \in F(r, 1) } P_{q - |s|}(d_1 - s(1), d_2 - s(2), \ldots, d_r - s(r), m).\]
\par Note that, in particular, this equation expresses $P_q$ in terms of $P_0, P_1, \ldots, P_{q-1}.$ Hence, to compute all of the $P_q,$ we turn to computing $P_0,$ allowing $m \neq 0.$
\begin{lemma}\label{lem: zeroth wedge characteristic}
We have that \[P_0(d_1, d_2, \ldots, d_r, m) = \frac{1}{n!} \sum\limits_{\epsilon \in F(r, 1)} \left(mH + \sum (d_i - \epsilon(i)) H_i\right)^n (-1)^{|\epsilon|},\] with the sum and product on the right-hand side being interpreted as the intersection product.
\end{lemma}
\begin{proof}
We use the short exact sequence $$0 \rightarrow \mathcal{I}_X \rightarrow \mathcal{O}_A \rightarrow i_*\mathcal{O}_X \rightarrow 0$$ along with Lemma \ref{lem: global Koszul}. We obtain the exact sequence \[0 \rightarrow \bigoplus\limits_{\substack{S \in \{1, 2, \ldots, r\}\\|S| = r}} \mathcal{O}_A(-\sum\limits_{i \in S} H_i) \rightarrow \bigoplus\limits_{\substack{S \in \{1, 2, \ldots, r\}\\|S| = r-1}} \mathcal{O}_A(-\sum\limits_{i \in S} H_i) \rightarrow \cdots \bigoplus\limits_{\substack{S \in \{1, 2, \ldots, r\}\\|S| = 1}} \mathcal{O}_A(-\sum\limits_{i \in S} H_i) \rightarrow \mathcal{O}_A \rightarrow i_*\mathcal{O}_X \rightarrow 0\] of sheaves (with all but the last being vector bundles). Now, tensor with $\mathcal{O}_A(d_1H_1 + d_2H_2 + \cdots d_rH_r + mH),$ and use the fact that the alternating sum of the Euler characteristics of these bundles sums to zero to find that \[\chi(A, i_*\mathcal{O}_X(d_1H_1 + \cdots + d_rH_r + mH)) = \sum\limits_{\epsilon \in F(r, 1)}(-1)^{|\epsilon|}\chi\left(A, \mathcal{O}_A\left(\sum\limits_{i=1}^r (d_i - \epsilon(i))H_i + mH\right)\right).\]
\par Now, we can use Riemann-Roch on abelian varieties (see \cite[\S 16]{mumford}) to compute that \[\chi(A, \mathcal{O}_A(D)) = D^n/n!\] for any divisor $D.$ Applying this to the above formula this yields \[\chi(A, i_*\mathcal{O}_X(d_1H_1 + \cdots + d_rH_r + mH)) = \frac{1}{n!} \sum\limits_{\epsilon \in F(r, 1)}(-1)^{|\epsilon|}\left(mH + \sum\limits_{i=1}^r (d_i - \epsilon(i))H_i\right)^n.\] Finally, since $i: X \rightarrow A$ is a closed immersion (so in particular, affine), and by the projection formula we have that $i_*\mathcal{O}_X \otimes \mathcal{O}_A(d_1H_1 + \cdots + d_rH_r + mH) = i_*(\mathcal{O}_X \otimes i^*\mathcal{O}_A(d_1H_1 + \cdots + d_rH_r + mH)),$ we have that \[\chi(A, i_*\mathcal{O}_X(d_1H_1 + \cdots + d_rH_r + mH)) = \chi(X, i^*\mathcal{O}_A(d_1H_1 + \cdots + d_rH_r + mH)) = P_0(d_1, d_2, \ldots, d_r, m).\] This proves the lemma.
\end{proof}
Having established the base case, we can now compute all of the polynomials, in the case where $m = 0.$ 
\begin{prop}\label{prop: euler characteristics}
We have the polynomial $P_q(d_1, d_2, \ldots, d_r, 0)$ equals \[\frac{1}{n!} \sum\limits_{\epsilon \in F(r, n), |\epsilon| = n} \binom{n}{\epsilon(1), \epsilon(2), \ldots, \epsilon(r)} H_1^{\epsilon(1)} H_2^{\epsilon(2)}\cdots H _r^{\epsilon(r)} c_{q; \epsilon(1), \epsilon(2), \ldots, \epsilon(n)}(d_1, d_2, \ldots, d_r),\] where \[c_{q; \epsilon(1), \epsilon(2), \ldots, \epsilon(n)}(d_1, d_2, \ldots, d_r) = \sum\limits_{s_1, s_2, \ldots, s_r \geq 0} (-1)^{\sum s_i}\binom{n}{q-\sum s_i} \prod\limits_{i=1}^r \left((d_i - s_i)^{\epsilon(i)} - (d_i - s_i - 1)^{\epsilon(i)}\right).\]
\end{prop}
\begin{proof}
We induct on $q.$ The case $q = 0$ follows from Lemma \ref{lem: zeroth wedge characteristic} by expanding out $\left(\sum\limits_{i=1}^r (d_i - \epsilon(i))H_i\right)^n$ using the multlinearity of the intersection product. 
\par Now, suppose that we've proven this for $q = 0, 1, \ldots, j-1.$ We wish to compute $P_j(d_1, d_2, \ldots, d_r).$ Using the recurrence relation, we find that \[c_{q; \epsilon(1), \epsilon(2), \ldots, \epsilon(r)}(d_1, d_2, \ldots, d_r) = \binom{n}{r} c_{0; \epsilon(1), \epsilon(2), \ldots, \epsilon(r)}(d_1, d_2, \ldots, d_r)\] \[ - \sum\limits_{s \in F(r, 1), |s| \neq 0} c_{q - |s|; \epsilon(1), \epsilon(2), \ldots, \epsilon(r)}(d_1 - s(1), d_2 - s(2), \ldots, d_r - s(r))\] for each $\epsilon \in F(r, n)$ where $|\epsilon| = n.$ Substituting the expression as above, we thus find that the coefficient is $c_{q; \epsilon(1), \epsilon(2), \ldots, \epsilon(r)}(d_1, d_2, \ldots, d_r)$ equal to \begin{align*} &- \sum\limits_{\substack{s \in F(r, 1) \\ |s| \neq 0}} \sum\limits_{t_1, t_2, \ldots, t_r \geq 0} (-1)^{\sum t_i}\binom{n}{q - |s| -\sum t_i} \prod\limits_{i=1}^r \left((d_i - s(i) - t_i)^{\epsilon(i)} - (d_i - t_i - s(i) - 1)^{\epsilon(i)}\right) \\ & + \binom{n}{r}\prod\limits_{i=1}^r \left(d_i^{\epsilon(i)} - (d_i - 1)^{\epsilon(i)}\right).\end{align*}
\par However, we can group the terms in the first summation together by the value of $(t_1 + s(1), t_2 + s(2), \ldots, t_n + s(n));$ then, note that for each such tuple $(x_1, x_2, \ldots, x_n)$ that the sum of those terms is $$\binom{n}{q - \sum x_i} (-1)^{\sum x_i} \prod\limits_{i=1}^r \left((d_i - x_i)^{\epsilon(i)} - (d_i - x_i - 1)^{\epsilon(i)}\right) \sum\limits_{\substack{s \in F(r, 1) \\ |s| \neq 0, s(i) \leq x_i}} (-1)^{\sum s(i)}.$$ But this sum is equal to $$-\binom{n}{q - \sum x_i} (-1)^{\sum x_i} \prod\limits_{i=1}^r \left((d_i - x_i)^{\epsilon(i)} - (d_i - x_i - 1)^{\epsilon(i)}\right)$$ unless $(x_1, x_2, \ldots, x_n) = 0,$ where it equals zero (as this tuple is never attained).
\par Combining these facts together yields that $c_{q; \epsilon(1), \epsilon(2), \ldots, \epsilon(r)}(d_1, d_2, \ldots, d_r)$ is \[\binom{n}{r}\prod\limits_{i=1}^r \left(d_i^{\epsilon(i)} - (d_i - 1)^{\epsilon(i)}\right) + \sum\limits_{\substack{x_1, x_2, \ldots, x_n \geq 0 \\ (x_1, x_2, \ldots, x_n) \neq 0}} \binom{n}{q - \sum x_i} (-1)^{\sum x_i} \prod\limits_{i=1}^r \left((d_i - x_i)^{\epsilon(i)} - (d_i - x_i - 1)^{\epsilon(i)}\right).\] But this is precisely the coefficient above that we wanted.
\end{proof}
Substituting in $d_1, d_2, \ldots, d_r = 0,$ we find that the coefficients in the above expression for $\chi(X, \Omega^q_{X/K})$ are \[(-1)^{n-r}\sum\limits_{s_1, s_2, \ldots, s_r \geq 0} (-1)^{\sum s_i}\binom{n}{q-\sum s_i} \prod\limits_{i=1}^r \left((s_i + 1)^{\epsilon(i)}-s_i^{\epsilon(i)} \right).\]
\par We now rewrite this in terms of Eulerian numbers.
\begin{prop}\label{prop: expression for cohomology in terms of Eulerian numbers}
The above expression is also equal to \[(-1)^{n-q-r} \sum\limits_{t \in F(r, q), |t| = q}\prod\limits_{i=1}^r E(\epsilon(i), t(i)).\]
\end{prop}
\begin{proof}
We expand out this expression using Proposition \ref{prop: alternating summation formula for Eulerian numbers}, before re-arranging this to get the above expression. First, note that we rewrite this summation as \begin{align*}E(n, k)  & = \sum\limits_{i=1}^{k+1} (-1)^{k + 1 - i} i^n \binom{n+1}{ k + 1 - i} & = \sum\limits_{i=1}^{k+1} (-1)^{k + 1 - i} i^n \left(\binom{n}{ k + 1 - i} + \binom{n}{k-i}\right) \\ & = \sum\limits_{i=0}^{k}(-1)^{k-i}\binom{n}{k-i}((i+1)^n - i^n)\end{align*} for $n \geq 1.$ For $n = 0$ we can see that both sides of the final equality we obtain are zero, so this final expression is still valid.
\par Substituting this in the above expression, we get a sum of terms that are a binomial coefficient times $\prod\limits_{i=1}^r ((s_i+1)^{\epsilon(i)} - s_i^{\epsilon(i)}).$ We can group these together by these products. For the term given by the product of $(s_i+1)^{\epsilon(i)} - s_i^{\epsilon(i)},$ we see that $E(\epsilon(i), t(i))$ contributes a coefficient of $(-1)^{t(i) - s_i} \binom{\epsilon(i)}{t(i) - s_i}.$ Thus, we may rewrite the given expression in the proposition statement as the following:
\[(-1)^{n-q-r} \sum\limits_{t \in F(r, q), |t| = q} \sum\limits_{s_1, s_2, \ldots, s_r \geq 0} (-1)^{q + \sum s_i} \prod\limits_{i=1}^r \binom{\epsilon(i)}{t(i) - s_i} ((s_i+1)^{\epsilon(i)} - s_i^{\epsilon(i)}).\] Swap the order of summation to get \[(-1)^{n-q-r}  \sum\limits_{s_1, s_2, \ldots, s_r \geq 0} \sum\limits_{t \in F(r, q), |t| = q}(-1)^{q + \sum s_i} \prod\limits_{i=1}^r \binom{\epsilon(i)}{t(i) - s_i} ((s_i+1)^{\epsilon(i)} - s_i^{\epsilon(i)}).\] We now consider the expression $\sum\limits_{t \in F(r, q), |t| = q} \prod\limits_{i=1}^r \binom{\epsilon(i)}{t(i) - s_i}.$ Note that this expression counts the number of ways to choose $\sum (t(i)-s_i) = q - \sum s_i$ items among a superset of $\sum\limits_{i=1}^r \epsilon(i) = n,$ so this just equals $\binom{n}{q - \sum s_i}.$ Substituting this in yields the expression above.
\end{proof}
In particular, observe that $(-1)^{n-q-r}\chi(X, \Omega^q_{X/K})$ is nonnegative. Furthermore, observe that from Lemma \ref{lem: generalized Eulerian in terms of Eulerian}, we see that if the $H_i$ are all equal to a given ample hypersurface class $H,$ $\chi(X, \Omega^q_{X/K}) = (-1)^{n-q-r}\frac{(H^n)}{n!} E_r(n, q).$
\subsection{Verifying the Numerical Conditions}
In order to show that the Shafarevich conjecture holds for complete intersections of hypersurfaces, we will need to verify a numerical condition on the Euler characteristics of $\Omega^q_{X/K}.$ Specifically, we will show the following.
\begin{prop}\label{prop: numerical condition}
Let $X$ be a smooth complete intersection of $r$ ample hypersurfaces in an abelian variety $A,$ defined over a number field $K.$ Then, we have that 
\[2\sum\limits_{q=0}^{n-r} |\chi(X, \Omega^q_{X/K})|^2 \leq (\sum\limits_{q=0}^{n-r} |\chi(X, \Omega^q_{X/K})|)^2.\]
\end{prop}
We first observe that it is enough to show that $2|\chi(X, \Omega^q_{X/K})| \leq \sum\limits_{q=0}^{n-r}|\chi(X, \Omega^q_{X/K})|$ for each $q.$ To see that this is enough, note that we then have that \[\sum\limits_{q=0}^{n-r} |\chi(X, \Omega^q_{X/K})|\left(2|\chi(X, \Omega^q_{X/K})| - \sum\limits_{q=0}^{n-r}|\chi(X, \Omega^q_{X/K})|\right) \leq 0,\] where re-arranging yields the desired inequality. We first prove the following.
\begin{lemma}\label{lem: most cases numerical condition}
If $X$ is not dimension $2$ or $4,$ we have $2|\chi(X, \Omega^q_{X/K})| \leq \sum\limits_{q=0}^{n-r}|\chi(X, \Omega^q_{X/K})|$ for each $q.$
\end{lemma}
\begin{proof}
It suffices, using Proposition \ref{prop: expression for cohomology in terms of Eulerian numbers} for $|\chi(X, \Omega^q_{X/K})| = (-1)^{n-q-r}\chi(X, \Omega^q_{X/K}),$ to prove that \[2\sum\limits_{\substack{t \in F(r, q) \\ |t| = q}}\prod\limits_{i=1}^r E(\epsilon(i), t(i)) \leq \sum\limits_{q=0}^{n-r} \sum\limits_{\substack{t \in F(r, q) \\ |t| = q}}\prod\limits_{i=1}^r E(\epsilon(i), t(i))\] for each choice of $\epsilon \in F(r, n)$ with $|\epsilon| = n,$ and for each $q.$ However, note that we can rewrite the desired inequality to be proven as \[2\sum\limits_{t(1), t(2), \ldots, t(r-1) \geq 0} E(\epsilon(r), q - \sum\limits_{i=1}^{r-1} t(i))\prod\limits_{i=1}^{r-1} E(\epsilon(i), t(i))\] \[\leq \sum\limits_{q=0}^{n-r} \sum\limits_{t(1), t(2), \ldots, t(r-1) \geq 0}  E(\epsilon(r), q - \sum\limits_{i=1}^{r-1} t(i))\prod\limits_{i=1}^{r-1} E(\epsilon(i), t(i)),\] since the only new terms that appear on both sides are zeroes (where we are taking Eulerian numbers $E(n, k)$ where $k \geq n$ or $k < n$). By this same logic, we can rewrite the right-hand side as just the sum $\sum\limits_{t(1), t(2), \ldots, t(r) \geq 0}  E(\epsilon(r), t(r))\prod\limits_{i=1}^{r-1} E(\epsilon(i), t(i)),$ since the only nonzero terms that will appear are those functions $t$ where $0 \leq t(i) < \epsilon(i)$ for each $i;$ in particular, we have $|t|$ lying between $0$ and $n - r$ inclusive.
\par We can hence rewrite the desired inequality as \[2\sum\limits_{\substack{t(1), t(2), \ldots, t(r) \geq 0 \\ \sum\limits_{i=1}^r t(i) = q}} E(\epsilon(r), t(r))\prod\limits_{i=1}^{r-1} E(\epsilon(i), t(i)) \leq \sum\limits_{t(1), t(2), \ldots, t(r) \geq 0}  E(\epsilon(r), t(r))\prod\limits_{i=1}^{r-1} E(\epsilon(i), t(i)),\]
\par We first claim that $2 E(n, k) \leq \sum\limits_{k=0}^{n-1} E(n, k)$ for each $n, k \geq 0$ (unless $(n, k) = (3, 1), (5, 2)$ or $(n, k) = (1, 0)$), since we can then apply this to $n = \epsilon(r)$ and $k = q - \sum\limits_{i=1}^{r-1} t(i).$ If $n = 0$ this is clear since both sides are zero. Furthermore, if $n$ is even, we know this since the Eulerian numbers are symmetric (and so there are two values of $k$ which both attain the maximum). Finally, suppose that $n \geq 7$ is odd. We note that the Eulerian numbers for $n=7$ are given by $1, 120, 1191, 2416, 1191, 120, 1,$ with the middle number $2416$ being less than half of the total, and for $n = 9$ we have that the Eulerian numbers are $$1, 502, 14608, 88234, 156190, 88234, 14608, 502, 1,$$ with the ratio of the middle number to the next-to-middle numbers being less than $2.$ 
\par From this, we claim that this remains true for larger $n.$ To prove this, we can use the recurrence relation for Eulerian numbers to find that \begin{align*}\frac{E(n+2, \frac{n-1}{2})}{E(n+2, \frac{n+1}{2})}  & = \frac{\frac{n+5}{2} E(n+1, \frac{n-3}{2}) + \frac{n+1}{2}E(n+1, \frac{n-1}{2})}{\frac{n+3}{2} E(n+1, \frac{n-1}{2}) + \frac{n+3}{2}E(n+1, \frac{n+1}{2})} \\ & = \frac{\frac{n+5}{2} \frac{n+5}{2} E(n, \frac{n-5}{2}) + \frac{n+5}{2} \frac{n-1}{2} E(n, \frac{n-3}{2}) + \frac{n+1}{2} \frac{n+3}{2} E(n, \frac{n-3}{2}) + \frac{n+1}{2} \frac{n + 1}{2}E(n, \frac{n-1}{2})}{\frac{n+3}{2} \frac{n+3}{2} E(n, \frac{n-3}{2}) + \frac{n+3}{2} \frac{n+1}{2} E(n, \frac{n-1}{2}) + \frac{n+3}{2} \frac{n+1}{2} E(n, \frac{n-1}{2}) + \frac{n+3}{2} \frac{n + 3}{2}E(n, \frac{n+1}{2})}.\end{align*} 
\par Recall that $E(n, \frac{n-1}{2})$ is the middle, and the Eulerian numbers are symmetric about $\frac{n-1}{2}.$ Thus, we can write this as \[\frac{(n+5)^2 \frac{E(n, \frac{n-5}{2})}{E(n, \frac{n-1}{2})} + (2n^2 + 8n - 2)\frac{E(n, \frac{n-3}{2})}{E(n, \frac{n-1}{2})} + (n+1)^2}{2(n+3)^2 \frac{E(n, \frac{n-3}{2})}{E(n, \frac{n-1}{2})} + 2(n+3)(n+1)},\] which we can bound above by \[\frac{1}{2} + \frac{(n^2 + 2n - 11)\frac{E(n, \frac{n-3}{2})}{E(n, \frac{n-1}{2})} - 2(n+1)}{2(n+3)^2 \frac{E(n, \frac{n-3}{2})}{E(n, \frac{n-1}{2})} + 2(n+3)(n+1)}\] (as we know that $E(n, \frac{n-5}{2}) > 0$). However, observe that $(n^2 + 2n - 11)\frac{E(n, \frac{n-3}{2})}{E(n, \frac{n-1}{2})} - 2(n+1) > 0$ for $n \geq 11.$ We can prove this by induction: indeed, note that $\frac{2(n+1)}{n^2 + 2n - 11} < \frac{1}{4}$ for $n \geq 9;$ so if $\frac{E(n, \frac{n-3}{2})}{E(n, \frac{n-1}{2})} > \frac{1}{2},$ then the same holds for $\frac{E(n + 2, \frac{n-1}{2})}{E(n+2, \frac{n+1}{2})} > \frac{1}{2}$ as well. In particular, we see that the largest Eulerian number $E(n, k)$ for any given $n,$ $E(n, \frac{n-1}{2}),$ is less than half of $E(n, \frac{n-1}{2}) + E(n, \frac{n-3}{2}) + E(n, \frac{n+1}{2}),$ and thus less than half of the sum of all $E(n, k)$ (ranging over $k$). But this is precisely what we want.
\par In particular, if there is some $\epsilon(i)$ which is not $1, 3,$ or $5$ for every term, then we are done. Indeed, we can start with the inequality \[E(\epsilon(i), t(i)) \leq \sum\limits_{t(i) \geq 0} E(\epsilon(i), t(i)),\] multiply both sides by the sum of the terms $\prod\limits_{\substack{1 \leq j \leq r \\ j \neq i}} E(\epsilon(j), t(j))$ over tuples $(t(j))_{j \neq i}$ whose sum is $q - t(i),$ then sum this expression over all $t(i).$
\par Next, suppose there are two $\epsilon(i) = 3, 5.$ We observe we can apply a similar trick, where we want to show that $$2 \sum\limits_{i+j = k} E(a, i)E(b, j) \leq \sum\limits_{i, j} E(a, i)E(b, j)$$ for each $k,$ with $a, b \in \{3, 5\}.$ But the values for $\sum\limits_{i+j = k} E(3, i)E(3, j)$ are given by $1, 8, 18, 8, 1,$ where we can see the inequality is satisfied. Similarly, we can check this for $(a, b) = (3, 5), (5, 3)$ and $(a, b) = (5, 5),$ as the numbers we obtain are $$1, 30, 171, 316, 171, 30, 1$$ for the first two and $$1, 52, 808, 3484, 5710, 3484, 808, 52, 1;$$ in both cases the largest number is less than half the total sum. 
\par Hence, the only cases we are left with is the case when $\epsilon(i) = 3$ or $5$ for one $i,$ and $1$ for the rest. In other words, we require that $n = r + 2$ or $n = r + 4,$ so our complete intersection is either a two-dimensional variety or a $4$-dimensional variety, which is what we wanted to check.
\end{proof}
For the case when $\dim X = 2, 4,$ we have to take another approach to verify the numerical condition. For $\dim X = 2,$ we see that the Euler characteristics are given by \begin{align*}\chi(X, \mathcal{O}_X) & = \sum\limits_{i = 1}^r (X \cdot H_i^2) \frac{1}{6} + \sum\limits_{1 \leq i < j \leq r} (X \cdot H_i \cdot H_j) \frac{1}{4} \\ \chi(X, \Omega^1_{X/K}) & = \sum\limits_{i = 1}^r (X \cdot H_i^2) \frac{4}{6} + \sum\limits_{1 \leq i < j \leq r} (X \cdot H_i \cdot H_j) \frac{2}{4} \\ \chi(X, \Omega^2_{X/K}) & = \sum\limits_{i = 1}^r (X \cdot H_i^2) \frac{1}{6} + \sum\limits_{1 \leq i < j \leq r} (X \cdot H_i \cdot H_j) \frac{1}{4}.\end{align*} If we let $a = \sum\limits_{i = 1}^r (X \cdot H_i^2) \frac{1}{6}$ and $b = \sum\limits_{1 \leq i < j \leq r} (X \cdot H_i \cdot H_j) \frac{1}{4},$ we see that the left hand side of \[2\sum\limits_{q=0}^{n-r} |\chi(X, \Omega^q_{X/K})|^2 \leq (\sum\limits_{q=0}^{n-r} |\chi(X, \Omega^q_{X/K})|)^2\] is $36a^2 + 40ab + 12b^2$ and the right-hand side is $36a^2 + 48ab + 16b^2,$ so the inequality holds.
\par For the case where $\dim X = 4,$ we then define the following variables: \begin{align*}
A & = \frac{1}{120}\sum\limits_{i=1}^r (X \cdot H_i^4), \quad B_{1, 3}  = \frac{1}{48}\sum\limits_{\substack{1 \leq i, j \leq r \\ i \neq j}} (X \cdot H_i \cdot H_j^3), \quad B_{2, 2} = \frac{1}{36}\sum\limits_{1 \leq i < j \leq r} (X \cdot H_i^2 \cdot H_j^2)  \\ C & = \frac{1}{24}\sum\limits_{\substack{1 \leq i, j, k \leq r \\ i < j, i, j \neq k}} (X \cdot H_i \cdot H_j \cdot H_k^2), \quad D = \frac{1}{16}\sum\limits_{1 \leq i < j < k < l \leq r} (X \cdot H_i \cdot H_j \cdot H_k \cdot H_l).
\end{align*}
We then observe that the Euler characteristics are given by \begin{align*}
\chi(X, \mathcal{O}_X) & = A + B_{1, 3} + B_{2, 2} + C + D\\
\chi(X, \Omega^1_{X/K}) & = 26A + 12B_{1, 3} + 8B_{2, 2} + 6C + 4D\\
\chi(X, \Omega^2_{X/K}) & = 66A + 22B_{1, 3} + 18B_{2, 2} + 10C + 6D\\
\chi(X, \Omega^3_{X/K}) & = 26A + 12B_{1, 3} + 8B_{2, 2} + 6C + 4D\\
\chi(X, \Omega^4_{X/K}) & = A + B_{1, 3} + B_{2, 2} + C + D.
\end{align*}
Again, the sum of the squares of these terms is given by the polynomial $$5710A^2 + 4156AB_{1, 3} + 774B_{1, 3}^2 + 3212AB_{2, 2} + 1180B_{1, 3}B_{2, 2} + 454B_{2, 2}^2 + 1948AC$$ $$ + 732B_{1, 3}C + 556B_{2, 2}C + 174C^2 + 1212AD + 460B_{1, 3}D + 348B_{2, 2}D + 220CD + 70D^2$$ and the square of the sum is 
$$14400A^2 + 11520AB_{1, 3} + 2304B_{1, 3}^2 + 8640AB_{2, 2} + 3456B_{1, 3}B_{2, 2} + 1296B_{2, 2}^2 + 5760AC $$ $$+ 2304B_{1, 3}C + 1728B_{2, 2}C + 576C^2 + 3840AD + 1536B_{1, 3}D + 1152B_{2, 2}D + 768CD + 256D^2,$$ and so in particular the coefficients of the latter are at least double of the former. In particular, we have finished the final cases needed to prove Proposition \ref{prop: numerical condition}.
\section{Reduction to the Combinatorial Statement}\label{section: reduce to combinatorics}
Having computed the appropriate Euler characteristics $\chi(X, \Omega^i),$ we are now ready to reduce Theorem \ref{thm: main theorem} to Theorem \ref{thm: big monodromy statement intro version}. We do this in two steps. First, we reduce the proof of the Shafarevich conjecture to Theorem \ref{thm: big monodromy main theorem}, using the output of the Lawrence-Venkatesh method as executed in \cite{km24}, stated in sufficiently generality. We will then utilize the procedure, as in \cite{ls25}, to make the reduction to the combinatorial statement.
\subsection{The Lawrence-Venkatesh Output}
We start by proving Theorem \ref{thm: main theorem}, assuming Theorem \ref{thm: big monodromy main theorem}. The main tool that we use is \cite[Theorem D]{km24}.
\par We start by constructing an appropriate scheme which classifies the smooth complete intersections of a given type. To obtain this scheme, suppose we have ample Neron-Severi classes $\phi_1, \phi_2, \ldots, \phi_r$ for hypersurfaces of an abelian variety $A,$ for a number field $K.$ Fix a very ample line bundle $L$ of $A.$
\par First, we note that all smooth complete intersections given by the data of these Neron-Severi classes all have the same Hilbert polynomial, with respect to the line bundle $L.$ To check, recall the Hilbert polynomial is given by $P_X(n) = \chi(X, \mathcal{O}_{X} \otimes L|_X).$ But we know from Lemma \ref{lem: zeroth wedge characteristic} that this Euler characteristic only depends on the Neron-Severi classes of the $H_i$ and $H$ appearing. In particular, for the complete intersections that we are given, fixing the Neron-Severi classes also fixes the Hilbert polynomial $\chi(X, \mathcal{O}_{X} \otimes L|_X).$ Let this Hilbert polynomial be $P.$
\par Now, consider the Hilbert scheme over $\mathcal{O}_{K, S}$ parametrizing closed subschemes of $A$ with Hilbert polynomial $P;$ this is a projective scheme over $\mathcal{O}_{K, S}$ (see \cite[Theorem 3.2, p.221-12]{fga}). Let $\mathcal{X} \rightarrow \Hilb_A^P$ be the universal subscheme of $\Hilb_A^P \times A.$ Our strategy will be to prove the finiteness of integral points on an appropriate subscheme of this Hilbert scheme. Specifically, we are interested in the following locus:
\begin{lemma}\label{lem: the locus we want}
Let $P \in \mathbb{Q}[z].$ Consider the subset of $\Hilb_{A_K}^P$ consisting of the image of geometric points $\Spec \Omega \rightarrow \Hilb_{A_K}^P$ such that the corresponding closed subscheme $\mathcal{X}_{\Omega}$ is smooth, geometrically connected, nondivisible, not a product, with ample normal bundle, and numerically like the complete intersection of $H_1, H_2, \ldots, H_r,$ satisfying the numerical conditions. Then, this subset is constructible. 
\end{lemma}
\par This is essentially \cite[Lemma 2.1]{km24}, except with the additional ``numerically like" condition. But this follows since applying this condition corresponds to taking appopriate connected components, as Euler characteristics are locally constant in flat families.
\par In particular, we can consider the corresponding subscheme, with the reduced structure, of this subset; let this subscheme be $\Hilb_{A_K}^{P, *},$ with $\mathcal{X}^* \rightarrow \Hilb_{A_K}^{P, *}$ the pullback of the universal subscheme to this constructible subset.
\par We now apply \cite[Theorem D]{km24}. First, note that we obtain the big monodromy condition by starting with the big Tannaka assumption given by Theorem \ref{thm: big monodromy main theorem}, using \cite[Theorem 4.10]{jklm23} and \cite[Theorem 6.1]{jklm23} to convert this statement to the necessary big monodromy condition. Next, we proceed as in the proof of \cite[Theorem 2.3]{km24}. A direct application of the proof there yields that there are only finitely many isomorphism classes of subvarieties of $A,$ with the given Hilbert polynomial $P,$ and which are smooth, nondivisible, with ample normal bundle, satisfying the numerical condition, geometrically connected, numerically like the appropriate complete intersection of hypersurfaces representing $d_i\phi,$ and with good reduction outside of $S.$ But from our work above, we see that the complete intersection of the $d_i \phi$ has ample normal bundle and satisfies the numerical condition. We thus obtain Theorem \ref{thm: main theorem}.
\par As such, in order to prove our main theorem, it suffices to prove Theorem \ref{thm: big monodromy main theorem}. 
\subsection{The Numerical Approach to Monodromy}\label{subsection: big monodromy}
We now reduce Theorem \ref{thm: big monodromy main theorem} to Theorem \ref{thm: big monodromy statement intro version}.
\par We start with the following theorem, combining two theorems from \cite{jklm23}.
\begin{thm}[Theorem A, Corollary 5.15 from \cite{jklm23}]\label{citedthm: simplicity of the Tannaka group}
Suppose that $X \subset A$ has ample normal bundle and is nondivisible. Then, if $\dim A \geq 3,$ the following are equivalent:
\begin{enumerate}
    \item $G_{X, \omega}^*$ is not simple,
    \item There are positive-dimensional smooth subvarieties $X_1, X_2 \subset A$ such that the sum morphism $\sigma: A \times A \rightarrow A$ induces an isomorphism $X_1 \times X_2 \simeq X.$ 
\end{enumerate}
Furthermore, if $G_{X, \omega}^*$ is simple, then the representation $\omega(\delta_X)$ is miniscule.
\end{thm}
Therefore, it follows that the pair of the image of $G_{X, \omega}^*$ in $\gl(\omega(\delta_X))$ and $\omega(\delta_X)$ must be one of the following pairs (see, for instance, \cite[Theorem 5.1.5]{Green_2013}):
\begin{enumerate}
    \item $\sll_n, \spp_n, \so_n$ acting on the standard representation,
    \item $\sll_n,$ acting on the $k$th wedge power of the standard representation for $2 \leq k \leq n-2,$
    \item $\text{Spin}_n,$ acting on either a spin representation or a half-spin representation (depending on the parity of $n$),
    \item $E_6$ or $E_7,$ acting on one of its lowest-dimensional irreducible representations.
\end{enumerate}
\par First, we remark that complete intersections of hypersurfaces cannot be expressed as products of subvarieties of $A.$ This is argued in \cite[Remark 6.3]{jklm23}.
\par Hence, to prove Theorem \ref{thm: big monodromy main theorem}, we need to show that only the first case occurs. The case that poses the most difficulty is the second one, taking $2 \leq k \leq n-2.$ The other two can be immediately ruled out simply by analyzing the dimension of our representation.
\begin{prop}\label{prop: all but the wedge representations ruled out}
Suppose that $X$ is numerically like a complete intersection of ample hypersurfaces $H_1, H_2, \ldots, H_r,$ and $\dim X > \frac{\dim A}{2}.$ Then, we cannot be in the third or fourth cases.
\end{prop}
\begin{proof}
We first recall from the classification of the third and fourth cases that the dimensions of the representations are either a power of $2,$ $27,$ or $56.$
\par We recall that the dimension of $\omega(\delta_X)$ is the Euler characteristic of $X,$ which in turn is given by \[\sum\limits_{i=0}^{\dim X} (-1)^i\chi(X, \Omega^i_{X/K}).\] But by using \cite[Lemma 2.14]{jklm23} (or by using Theorem \ref{prop: alternating summation formula for Eulerian numbers} and summing these up), we see that the absolute value of the Euler characteristic of $X$ is \[\sum\limits_{\substack{e \in F(r, n), |e| = n\\ \forall i, \space e(i) \geq 1}} H_1^{e(1)} H_2^{e(2)} \cdots H_r^{e(r)}.\] 
\par However, observe that $H_i^{e(i)}$ is divisible by $e(i)!$ in the Chow group (for instance, viewing $H_i^{e(i)} \in H^{2e(i)}(A, \mathbb{Z}) = \bigwedge^{2e(i)} H^1(A, \mathbb{Z})$ and recalling that the intersection product corresponds to wedge products). Therefore, in particular, since $r < \frac{n}{2},$ we see that some $e(i)$ is at least $3$ for each term. But this means that every term in this summation is divisible by $6,$ and thus $\chi_{\text{top}}(X)$ is divisible by $6.$ In particular, it cannot be a power of $2,$ $27,$ or $56,$ which is what we wanted to check.
\end{proof}
Furthermore, once we rule out the wedge product of the standard representation, we note that the appropriate cases for the standard representation occur in the situations required. This follows by an argument analogous to \cite[Lemma 3.11]{ls25}; the group $G_{X, \omega}^*$ being either $\spp_n$ or $\so_n$ implies the existence of an appropriate bilinear form, which via the perverse sheaf formalism translates to the appropriate symmetry conditions on $X.$
\par Hence, to prove that we have big monodromy for most torsion tuples, we need to show that the wedge powers do not occur. To do this, we have the following analog of \cite[Lemmas 3.13, 3.14]{ls25}.
\begin{lemma}\label{lem: numerical approach foundation}
Suppose $X$ is smooth and nondivisible, and numerically like a complete intersection of hypersurfaces in an abelian variety $A,$ with $n = \dim A$ and $r$ the codimension of $X$ in $A.$ Suppose that $G_{X, \omega}^*, \omega(\delta_X)$ is given by the second case in the classification above, as the $k$th wedge power of the standard representation of $\sll_m.$ Then, there exists a function $m_H: \mathbb{Z} \rightarrow \mathbb{N}$ and an integer $s$ such that $\sum\limits_{i} m_H(i) = m$ and, for all $q \in \mathbb{Z},$
\[\sum\limits_{\substack{m_S: \mathbb{Z} \rightarrow \mathbb{Z} \\ 0 \leq m_S(i) \leq m_H(i) \\ \sum m_S(i) = k \\ \sum im_S(i) = s + q}} \prod\limits \binom{m_H(i)}{m_S(i)} = |\chi(X, \Omega^q_{X/K})|.\]
\end{lemma}
The proof of this lemma follows exactly as in \cite{ls25}, except we replace our hypersurface with $X,$ and replacing $|\chi(H, \Omega^q_{H/K})|$ with $|\chi(X, \Omega^q_{X/K})|.$ Indeed, the only things we need are the Tannakian formalism applied to $\langle \delta_X \rangle$ (which we have outlined above), and the comparison theorems from $p$-adic Hodge theory (which we can use as $X$ is smooth and projective over $K$).
\par Using this lemma, we now spend the remainder of the paper showing this equation has no solutions in the situations that we need for our theorem.
\section{Asymptotic Proof of Big Monodromy}\label{section: big monodromy asymptotics}
Our goal is now to prove the following.
\begin{thm}\label{thm: big monodromy statement}
Suppose that $X$ is numerically like the complete intersection of $r$ hypersurfaces in $A,$ all representing multiples of the same ample Neron-Severi class $H,$ namely by $d_1H, d_2H, \ldots, d_rH.$ Then, if $\dim A = n \geq 10r^4 + 1000,$ there is no solution in the functions $m_H(w),$ positive integers $k$ between $2$ and $m - 1$ (where $m = \sum\limits_i m_H(i)$), and integer $s,$ to the following system of equations indexed by integers $q$ between $0$ and $n-r,$ respectively:
\[\sum\limits_{\substack{m_S: \mathbb{Z} \rightarrow \mathbb{Z} \\ 0 \leq m_S(i) \leq m_H(i) \\ \sum m_S(i) = k \\ \sum im_S(i) = s + q}} \prod\limits \binom{m_H(i)}{m_S(i)} = |\chi(X, \Omega^q_{X/K})|.\]
\par Furthermore, if the $d_i$ are all equal, there is no solution for $n \geq 10r^2 + 1000.$
\end{thm}
In particular, this is enough to prove that we have big monodromy in the specified ranges (subject to the assumptions specified earlier), and hence prove the Shafarevich conjecture for complete intersections of $r$ hypersurfaces in those cases. We also know that the above theorem holds true by \cite[Proposition B.1]{ls25} in the case $r = 1,$ so we may assume that $r \geq 2.$
\par We proceed similarly as in \cite[Appendix B]{ls25}; we will first make some preliminary descriptions on the form that $m_H$ takes (as well as the form that $m_S$ takes for certain values of $q$). Our approach, in broad strokes, will then be to show there is no solution via asymptotics, by arguing that on the left-hand side, as $n$ grows, either a given term is too big, or it is zero and the remaining terms are too small. 
\par From here on out, assume for the sake of contradiction that such a positive integer $k,$ integer $s,$ and function $m_H: \mathbb{Z} \rightarrow \mathbb{Z}$ exists.
\subsection{Preliminary Control of $m_H$}
We begin the proof with a few preliminary estimates. First are a few lemmas from \cite{ls25} which we will find useful to rewrite some of the equations in the above theorem.
\par Let $m_0, m_{\text{max}}$ be the functions satisfying $\sum m_S(i) = k$ and $0 \leq m_S(i) \leq m_H(i)$ which minimize and maximize $\sum im_S(i),$ respectively. In general, we will let $[x]$ be the function which is $1$ on $x$ and zero everywhere else. So, for instance, two of the functions for $q = 1$ are $m_0 + [w] - [w-1]$ and $m_0 + [w+1] - [w].$
\begin{lemma}[Lemma B.2 from \cite{ls25}]\label{lem: w and w' cutoffs}
There exists a unique integer $w$ such that $m_0(i) = m_H(i)$ for $i < w,$ $m_0(i) = 0$ for $i > w,$ and $m_0(w) > 0.$ Furthermore, there exists a unique integer $w'$ such that $m_{\text{max}}(i) = m_H(i)$ for $i > w',$ $m_{\text{max}}(i) = 0$ for $i < w',$ and $m_{\text{max}}(w') > 0.$ Furthermore, $w \leq w'.$
\end{lemma}
\begin{lemma}[Lemma B.3 from \cite{ls25}]\label{lem: basic summation for mH}
We have the equality \[\sum\limits_{i < w} (w-i) m_H(i) + \sum\limits_{i > w'} (i-w')m_H(i) + k(w' - w) = n - r.\]
\end{lemma}
The only reason for the difference between the appearance of this lemma and the original is that the range for $q$ above is between $0$ and $n-r,$ rather than $0$ and $n-1.$ The proofs of these lemmas are the same as their counterparts in \cite{ls25}.
\par Next, we show that $m_0(w)$ cannot be too small, if we have solution $m_H(w).$ We utilize here that $r \geq 2,$ since then $\chi(X, \mathcal{O}_X)$ is guaranteed to be comparatively large regardless of the hypersurfaces involved. 
\begin{lemma}\label{lem: m0 not 1.}
We have $m_0(w) \geq 2$ and $m_H(w) - m_0(w) \geq 2.$
\end{lemma}
\begin{proof}
Assume for the sake of contradiction that $m_0(w) = 1.$ Then, the equations for $q = 0, 1$ read as follows:
\[m_H(w) = |\chi(X, \mathcal{O}_X)|,\] \[m_H(w-1)\binom{m_H(w)}{2} + m_H(w+1) = |\chi(X, \Omega^1_{X/K})|,\] for instance.
\par Let $s$ be the smallest positive integer such that $m_H(w-s) \neq 0.$ In particular, we have that $m_H(w-1) = m_H(w-2) = \cdots = m_H(w-s-1) = 0.$ We claim that $s \geq r - 1.$
\par To prove this, note that for this value of $s$ that one of the terms for $q = s$ is $m_H(w-s) \binom{m_H(w)}{2},$ with the corresponding term being $|\chi(X, \Omega^{s}_{X/K})|.$ Dividing by $m_H(w) = |\chi(X, \mathcal{O}_X)|,$ we thus know for this value of $s$ that \[m_H(w-s) \frac{(m_H(w) - 1)}{2} \leq \left|\frac{\chi(X, \Omega^s_{X/K})}{\chi(X, \mathcal{O}_X)}\right|.\] But we know that $m_H(w-s) \geq 1$ and $m_H(w) = |\chi(X, \mathcal{O}_X)|.$ Hence, for this value of $s,$ we need $|\chi(X, \Omega^s_{X/K})| \geq |\chi(X, \mathcal{O}_X)|\frac{|\chi(X, \mathcal{O}_X)| - 1}{2}.$ We will prove that in Lemma \ref{lem: inequality to show mH(w-s) not 0 means s large} that this implies that $s \geq r - 1.$
\par From here, note that this implies that $m_H(w+i) = |\chi(X, \Omega_{X/K}^i)|$ for $i = 1, 2, \ldots, s-1.$ However, note that if we look at the equation for $3s - 2,$ we have a term that looks like $m_H(w-s) \binom{m_H(w + s - 1)}{2}.$ We thus arrive at the inequality \[|\chi(X, \Omega_{X/K}^{3s-2})| \geq \frac{|\chi(X, \Omega_{X/K}^{s-1})|(|\chi(X, \Omega_{X/K}^{s-1})|-1)}{2},\] which we show cannot hold in Lemma \ref{lem: large s means one term in 3s-2 region is too big}. This is a contradiction, and thus shows that $m_0(w) \neq 1.$
\par Symmetrically, observe that if $m_H(w) - m_0(w) = 1,$ we would get the equations \[m_H(w) = |\chi(X, \mathcal{O}_X)|,\] \[m_H(w+1)\binom{m_H(w)}{2} + m_H(w-1) = |\chi(X, \Omega^1_{X/K})|,\] where the same argument as above (swapping $m_H(w-i)$ and $m_H(w+i)$ in the above argument) would yield a contradiction.
\end{proof}
\par We now split up our argument into two cases: the case where $w \neq w',$ and the case where $w = w'.$ We will start by giving the main outline of the argument, reducing it down to explicit inequalities involving the polynomials that we found in Section \ref{section: cohomology computation}. We then save the proof of these inequalities for the appendix.
\par In particular, note that if one can prove the inequalities in the appendix for other families subvarieties of $A,$ then one obtains the desired big monodromy statement for that family of subvarieties of $A.$ 
\subsection{The $w \neq w'$ case}
We first assume that $w \neq w'.$ In this case, we know that, because of the equalities above, $k \leq n - r,$ and therefore $m_0(w) \leq n - r.$
\par We begin by looking at the ratio of the equations for $|\chi(X, \mathcal{O}_X)|$ and $|\chi(X, \Omega^1_{X/K})|,$ yielding the following equation: \[m_H(w-1) \frac{m_H - m_0}{m_0 + 1} + m_H(w+1) \frac{m_0}{m_H - m_0 + 1} = \left|\frac{\chi(X, \Omega^1_{X/K})}{\chi(X, \mathcal{O}_X)}\right|.\] In particular, we know that at least one of these terms is at least $\left|\frac{\chi(X, \Omega^1_{X/K})}{2\chi(X, \mathcal{O}_X)}\right|.$ We can do a casework on which of these terms satisfies this inequality to show that both of them yield a contradiction.
\begin{lemma}\label{lem: m0 cannot be too big}
In the situation given, we must have $$\frac{m_0(w)}{m_H(w) - m_0(w) + 1} < 4.$$ 
\end{lemma}
\begin{proof}
We start with the equation $\binom{m_H(w)}{m_0(w)} = |\chi(X, \mathcal{O}_X)|.$ However, looking at the formula for this Euler characteristic, Proposition \ref{prop: euler characteristics}, we see that $|\chi(X, \mathcal{O}_X)| \geq \frac{(H)^n}{n!}E_r(n, 0).$ But we know that the term $\frac{(H)^n}{n!}$ is a positive integer, using Riemann-Roch (again, see \cite[\S 16]{mumford}) and the fact that $H$ is ample. 
\par It thus follows that $\binom{m_H(w)}{m_0(w)} \geq E_r(n, 0).$ At the same time, note that $E_r(n, 0) \geq E_2(n-r + 2, 0):$ we can get a permutation of $n$ elements and $r - 1$ indices to ignore by taking a permutation of $n - r + 2$ with one ignored position, then appending the numbers $n - r + 3, n - r + 4, \ldots, n.$ 
\par Thus, we have $\binom{m_H(w)}{m_0(w)} \geq E_r(n-r+2, 0) \geq 2^{n-r+2} - (n-r + 3),$ the last inequality coming from Lemma \ref{lem: Eulerian asymptotic}.
\par However, if $\frac{m_0(w)}{m_H(w) - m_0(w) + 1} \geq 4,$ then $4m_H(w) \leq 5m_0(w).$ Therefore, we have that $\binom{m_H(w)}{m_0(w)} \leq \binom{\lceil \frac{5}{4}m_0(w) \rceil}{m_0(w)} = \binom{\lceil \frac{5m_0(w)}{4}\rceil}{\lceil \frac{m_0(w)}{4}\rceil}$ However, observe that $n! > (n/e)^n,$ for instance; we can use induction, observing that $2 > (2/e)^2$ and that $\frac{(n+1)!}{n!} = (n+1) \geq (n+1) \frac{(1 + 1/n)^n}{e} = \frac{((n+1)/e)^{n+1}}{(n/e)^n}$). We can therefore bound the above by \[(\frac{1}{e}\lceil \frac{m_0(w)}{4}\rceil)^{-\lceil \frac{m_0(w)}{4}\rceil} (\lceil \frac{5}{4}m_0(w) \rceil)^{\lceil \frac{m_0(w)}{4}\rceil}.\] But we can check that \[\frac{\lceil \frac{5}{4}m_0(w) \rceil}{\lceil \frac{m_0(w)}{4}\rceil} \leq 5,\] so hence the above is at most $(5e)^{\lceil \frac{m_0(w)}{4}\rceil} \leq (4e)^{\frac{n-r+3}{4}}.$ But $5e < 14 < 2^4;$ therefore, we can bound this above by $(14)^{(n-r+3)/4}.$
\par In other words, we need $2^{n-r+2} - (n-r+3) \leq (14)^{(n-r+3)/4}.$ Note that the increasing $n$ by $1$ scales the left-hand side up by a factor more than $2,$ but the right-hand side by at most $14^{1/4} < 2;$ it thus suffices to find one such $n$ where this inequality holds.
\par But we can take $n = r + 21;$ we then have that $14^6 = 7529536,$ but $2^{23} - 24 = 8388584,$ so hence we see that this inequality is satisfied for $n \geq r + 21,$ and in particular for the $n$ ranges that we specified in Theorem \ref{thm: big monodromy statement}.
\end{proof}
\begin{lemma}
For the ranges of $n$ specified above, we cannot have $m_H(w+1) \frac{m_0}{m_H - m_0 + 1} \geq \left|\frac{\chi(X, \Omega^1_{X/K})}{2\chi(X, \mathcal{O}_X)}\right|.$
\end{lemma}
\begin{proof}
First, note that if the inequality given above were true, then we'd have that $m_H(w+1) > \left|\frac{\chi(X, \Omega^1_{X/K})}{8 \chi(X, \mathcal{O}_X)}\right|;$ in particular, by Lemma \ref{lem: the respective ratios are big enough}, this is larger than $2.$ However, note that one of the terms arising in the equation for $q = 2$ (after dividing out by the zeroth equation) is given by $\binom{m_H(w+1)}{2} \frac{m_0(w)}{m_H(w) - m_0(w) + 1} \frac{m_0(w) - 1}{m_H(w) - m_0(w) + 2};$ this should be at most $\left|\frac{\chi(X, \Omega^2_{X/K})}{\chi(X, \mathcal{O}_X)}\right|.$
\par However, note that $\frac{\frac{m_0(w) - 1}{m_H(w) - m_0(w) + 2}}{\frac{m_0(w)}{m_H(w) - m_0(w) + 1}} \geq \frac{1}{3},$ as $m_0(w), m_H(w) - m_0(w) + 1 \geq 2.$ Therefore, the above implies that $$\frac{1}{6}\frac{m_H(w+1) - 1}{m_H(w+1)} \left(m_H(w+1) \frac{m_0(w)}{m_H(w) - m_0(w) + 1}\right)^2 \leq \left|\frac{\chi(X, \Omega^2_{X/K})}{\chi(X, \mathcal{O}_X)}\right|.$$ In other words, as we know that $m_H(w+1) \geq 2$ in this case, we would need \[\frac{1}{48} \left(\frac{\chi(X, \Omega^2_{X/K})}{\chi(X, \mathcal{O}_X)}\right)^2 \leq \left|\frac{\chi(X, \Omega^2_{X/K})}{\chi(X, \mathcal{O}_X)}\right|.\] But this contradicts Lemma \ref{lem: the core q=2 inequality}.
\end{proof}
This handles the first case. For the second, we can utilize the fact that $m_H(w-1) \leq n-r$ by Lemma \ref{lem: basic summation for mH} to obtain our desired contradiction.
\begin{lemma}\label{lem: mH(w-1) term big}
For the ranges of $n$ specified above, we cannot have $m_H(w-1) \frac{m_H(w) - m_0(w)}{m_0(w) + 1} \geq \left|\frac{\chi(X, \Omega^1_{X/K})}{2\chi(X, \mathcal{O}_X)}\right|.$
\end{lemma}
\begin{proof}
Again, we prove this by contradiction. First, suppose that $m_H(w-1) \geq 2.$ We can again look at the $q = 2$ equation. This time, the term that we are interested in is $\binom{m_H(w-1)}{2} \frac{m_H(w) - m_0(w)}{m_0(w) + 1} \frac{m_H(w) - m_0(w) - 1}{m_0(w) + 2};$ again, this should be at most $\left|\frac{\chi(X, \Omega^2_{X/K})}{\chi(X, \mathcal{O}_X)}\right|.$
\par This time, observe that $\binom{m_H(w-1)}{2} \geq \frac{m_H(w-1)^2}{4},$ since $m_H(w-1) \geq 2$ by assumption. Furthermore, note that $m_H(w) - m_0(w) \geq 2,$ by Lemma \ref{lem: m0 not 1.}. Therefore, the above implies that $\frac{1}{12} \left(m_H(w-1) \frac{m_H(w) - m_0(w)}{m_0(w) + 1}\right)^2 \leq \left|\frac{\chi(X, \Omega^2_{X/K})}{\chi(X, \mathcal{O}_X)}\right|.$ Substituting in our assumption in the statement of the lemma, we would thus require \[\frac{1}{48} \left(\frac{\chi(X, \Omega^1_{X/K})}{\chi(X, \mathcal{O}_X)}\right)^2 \leq \left|\frac{\chi(X, \Omega^2_{X/K})}{\chi(X, \mathcal{O}_X)}\right|.\] But this contradicts Lemma \ref{lem: the core q=2 inequality} again.
\par Therefore, we have that $m_H(w-1) = 1.$ In this case, we will now show that the other terms in the $q=2$ equation are too small. Specifically, we can look at the equation, divided by the $q = 0$ equation (and recalling that $m_H(w-1)$ term is assumed to be zero here):
\begin{align*} \left|\frac{\chi(X, \Omega^2_{X/K})}{\chi(X, \mathcal{O}_X)}\right|  & = m_H(w-2) \frac{m_H(w) - m_0(w)}{m_0(w) + 1} + m_H(w+2)\frac{m_0(w)}{m_H(w) - m_0(w) + 1} + m_H(w-1)m_H(w+1) \\ & + \frac{m_H(w+1)(m_H(w) - 1)}{2}\frac{m_0(w)(m_0(w) - 1)}{(m_H(w) - m_0(w) + 1)(m_H(w) - m_0(w) + 2)}.\end{align*}
\par To prove this, we follow the arguments of \cite[Lemma B.13]{ls25}, \cite[Lemma B.14]{ls25}. For the first term, we use the simple bound that $m_H(w-2) \leq \frac{n-r}{2}$ from Lemma \ref{lem: basic summation for mH}; hence, this term is bounded above by $\frac{n-r}{2} \left|\frac{\chi(X, \Omega^1_{X/K})}{\chi(X, \mathcal{O}_X)}\right|.$ For the second term, we can look at the $q = 5$ equation, divided by the $q = 0$ equation. One of the terms arises from the function $m_0(w) + 2[w+2] - [w] - [w-1],$ yielding the term $$\frac{m_H(w+2)(m_H(w+2) - 1)m_H(w-1)}{2} \frac{m_0(w)}{m_H(w) - m_0(w) + 1};$$ this must be at most $\left|\frac{\chi(X, \Omega^5_{X/K})}{\chi(X, \mathcal{O}_X)}\right|.$ Multiplying by the term $m_H(w-1) \frac{m_H(w) - m_0(w)}{m_0(w) + 1} \leq \frac{\chi(X, \Omega_{X/K}^1)}{\chi(X, \mathcal{O}_X)},$ since we are assuming $m_H(w-1) = 1,$ we get the inequality $$\frac{m_H(w+2)(m_H(w+2) - 1)}{2} \frac{m_0(w)(m_H(w) - m_0(w))}{(m_0(w) + 1)(m_H(w) - m_0(w) + 1)} \leq \left|\frac{\chi(X, \Omega_{X/K}^5)\chi(X, \Omega_{X/K}^1)}{\chi(X, \mathcal{O}_{X})^2}\right|.$$ We see that $m_H(w) - m_0(w) \geq 2$ from the above, so this second fraction is bounded from below by $\frac{4}{9},$ as we know that $m_0(w) \geq 2.$
\par Therefore, we have that $$m_H(w+2)(m_H(w+2) - 1) \leq \left|\frac{9\chi(X, \Omega_{X/K}^5)\chi(X, \Omega_{X/K}^1)}{2\chi(X, \mathcal{O}_{X})^2}\right|.$$ By the same logic, but now looking at $q = 3,$ we have that $$m_H(w+1)(m_H(w+1) - 1) \leq \left|\frac{9\chi(X, \Omega_{X/K}^3)\chi(X, \Omega_{X/K}^1)}{2\chi(X, \mathcal{O}_{X})^2}\right|,$$ giving us control over the fourth term.
\par We now combine these together. Recall that, if $m_H(w-1) = 1,$ then $\frac{m_H(w) - m_0(w)}{m_0(w) + 1} \geq \left|\frac{\chi(X, \Omega^1_{X/K})}{2\chi(X, \mathcal{O}_X)}\right|,$ and thus $\frac{m_0(w)}{m_H(w) - m_0(w) + 1}, \frac{m_0(w) - 1}{m_H(w) - m_0(w) + 2} \leq \left|\frac{2\chi(X, \mathcal{O}_X)}{\chi(X, \Omega^1_{X/K})}\right|.$ We then have the inequality \begin{align*}\left|\frac{\chi(X, \Omega^2_{X/K})}{\chi(X, \mathcal{O}_X)}\right| & \leq \frac{n-r}{2} \left|\frac{\chi(X, \Omega^1_{X/K})}{\chi(X, \mathcal{O}_{X})}\right| + \left|\frac{2\chi(X, \mathcal{O}_X)}{\chi(X, \Omega^1_{X/K})}\right|\left(\sqrt{\left|\frac{9\chi(X, \Omega_{X/K}^5)\chi(X, \Omega_{X/K}^1)}{2\chi(X, \mathcal{O}_{X})^2}\right|} + 1\right)\\ & + \sqrt{\left|\frac{9\chi(X, \Omega_{X/K}^3)\chi(X, \Omega_{X/K}^1)}{2\chi(X, \mathcal{O}_{X})^2}\right|} + 1 + \left(\frac{2\chi(X, \mathcal{O}_X)}{\chi(X, \Omega^1_{X/K})}\right)^2 \left(\left|\frac{9\chi(X, \Omega_{X/K}^3)\chi(X, \Omega_{X/K}^1)}{4\chi(X, \mathcal{O}_{X})^2}\right|\right).\end{align*}
\par In other words, since we know that $|\chi(X, \mathcal{O}_X)| < |\chi(X, \Omega^1_{X/K})|,$ we have \begin{align*}\left|\frac{\chi(X, \Omega^2_{X/K})}{\chi(X, \mathcal{O}_X)}\right| & \leq \frac{n-r}{2} \left|\frac{\chi(X, \Omega^1_{X/K})}{\chi(X, \mathcal{O}_{X})}\right| + \left(\sqrt{\left|\frac{18\chi(X, \Omega_{X/K}^5)}{\chi(X, \Omega_{X/K}^1)}\right|}\right) + \sqrt{\left|\frac{9\chi(X, \Omega_{X/K}^3)\chi(X, \Omega_{X/K}^1)}{2\chi(X, \mathcal{O}_{X})^2}\right|} \\ & + \left(\left|\frac{9\chi(X, \Omega_{X/K}^3)}{\chi(X, \Omega_{X/K}^1)}\right|\right) + 3.\end{align*}
We show in Lemma \ref{lem: remaining terms in q=2} that this doesn't hold, which is a contradiction.
\end{proof}
But we know that one of these inequalities must hold. Hence, if Theorem \ref{thm: big monodromy statement} were to be false, we would need to have $w = w'.$
\subsection{The $w = w'$ case}
Our main tool in this case is the ratio of the $q = 1$ and $q = 0$ equations, since knowing $w = w'$ means that, by Lemma \ref{lem: basic summation for mH}, $m_H(w-i) \leq n -r$ for $i \neq 0,$ which gives us significantly more control over the terms. Our argument proceeds very similarly to Lemma \ref{lem: mH(w-1) term big}.
\par In this case, we know that $m_H(w - i) \leq n-r$ for all $i \neq 0$ by Lemma \ref{lem: basic summation for mH}. We now can use the same dichotomy that we used in the $w \neq w'$ case. Namely, we consider which of the two terms on the left-hand side of this equation is bigger:
\[m_H(w-1) \frac{m_H(w) - m_0(w)}{m_0(w) + 1} + m_H(w+1) \frac{m_0(w)}{m_H(w) - m_0(w) + 1} = \left|\frac{\chi(X, \Omega^1_{X/K})}{\chi(X, \mathcal{O}_X)}\right|.\] As we know that $w = w',$ and the equation for Lemma \ref{lem: basic summation for mH} is symmetric, we can without loss of generality assume that the first term is the larger one. Hence, $m_H(w-1) \frac{m_H(w) - m_0(w)}{m_0(w) + 1} \geq \left|\frac{\chi(X, \Omega^1_{X/K})}{2\chi(X, \mathcal{O}_X)}\right|.$
\par Again, we assume first that $m_H(w-1) \geq 2,$ so one of the nonzero terms in the equation for $q = 2$ (divided by the $q=0$ equation) is $$\binom{m_H(w-1)}{2} \frac{m_H(w) - m_0(w)}{m_0(w) + 1} \frac{m_H(w) - m_0(w) - 1}{m_0(w) + 2}.$$ Again, since $m_H(w-1) \geq 2,$ the binomial is at least $m_H(w-1)^2/4.$ On the other hand, we claim that $\frac{m_H(w) - m_0(w) - 1}{m_0(w) + 2} \geq \frac{1}{3}\frac{m_H(w) - m_0(w)}{m_0(w) + 1};$ this follows since $m_0 \geq 2,$ so $\frac{m_0(w) + 1}{m_0(w) + 2} \geq \frac{2}{3}.$ Meanwhile, $m_H(w) - m_0(w) \geq 2,$ using Lemma \ref{lem: m0 not 1.}, and the fact that $m_H(w-1) \leq n-r,$ so we have the inequality $\frac{m_H(w) - m_0(w) - 1}{m_H(w) - m_0(w)} \geq \frac{1}{2},$ yielding the desired inequality. 
\par Combining everything, we obtain $$\frac{1}{12} \left(m_H(w-1) \frac{m_H(w) - m_0(w)}{m_0(w) + 1}\right)^2 \leq \left|\frac{\chi(X, \Omega^2_{X/K})}{\chi(X, \mathcal{O}_X)}\right|$$ to be satisfied, if a solution for $m_H$ exists. But again this cannot be the case by Lemma \ref{lem: the core q=2 inequality}.
\par We can finish by the exact argument as in Lemma \ref{lem: mH(w-1) term big}, since we know that $m_H(w) - m_0(w) \geq 2$ and $m_0(w) \geq 2$ by Lemma \ref{lem: m0 not 1.}. Having covered all the cases, we obtain a contradiction in each one, therefore finishing the proof of Theorem \ref{thm: big monodromy statement}.
\appendix
\section{Appendix: Asymptotic Inequalities}
In this appendix, we will prove the inequalities that we used in the previous section. Note that all of the inequalities featured here hold in the appropriate limit by considering asymptotic behavior. Indeed, in the case where $X$ is the complete intersection of $r$ hypersurfaces which all represent the same ample Neron-Severi class $H,$ we know that $\chi(X, \Omega^q_{X/K}) = \frac{(H)^n}{n!} E_r(n, q).$ But from Lemma \ref{lem: Eulerian asymptotic}, this is roughly $(r + q)^n,$ and so for sufficiently large $n$ all these inequalities should hold by substituting in this asymptotic for the Eulerian numbers. 
\par The main tasks that need to be done in the proofs in this subsection are as follows: to pin down what sufficiently large $n$ means, and to show that this heuristic still holds in more general settings.
\subsection{Inequalities Featured in Lemma \ref{lem: m0 not 1.}}
\begin{lemma}\label{lem: inequality to show mH(w-s) not 0 means s large}
If $X$ is the complete intersection of ample hypersurfaces $H_1, H_2, \ldots, H_r$ in an abelian variety of dimension $n \geq 3r^2 + 100,$ then $|\chi(X, \Omega^s_{X/K})| \geq |\chi(X, \mathcal{O}_X)|\frac{|\chi(X, \mathcal{O}_X)| - 1}{2}$ implies that $s \geq r - 1.$
\end{lemma}
\begin{proof}
We expand out the left-hand side and the first term in the product of the right-hand side, using the polynomials given by Proposition \ref{prop: expression for cohomology in terms of Eulerian numbers}. Comparing coefficients for $(H_1^{e_1}H_2^{e_2} \cdots H_r^{e_r}),$ and noting that if one of the $e_i$ is zero then both coefficients are zero, it suffices to prove for each choice of $e_i,$ whose sum is $n$ and where none of them are zero, that \[\frac{|\chi(X, \mathcal{O}_X)| - 1}{2} \geq \sum\limits_{t \in F(r, s), |t| = s} \prod\limits_{i=1}^r E(e_i, t(i))\] for $s \leq r - 2.$
\par However, using Lemma \ref{lem: Eulerian asymptotic}, we note that $E(e_i, t(i)) \leq (t(i) + 1)^{e_i},$ so this is bounded above by $$\sum\limits_{t \in F(r, s), |t| = s} \prod\limits_{i=1}^r (t(i) + 1)^{e_i}.$$ By the AM-GM inequality, $\prod\limits_{i=1}^r (t(i) + 1)^{e_i} \leq \left(1 + \sum\limits_{i=1}^r \frac{e_i}{n} t(i) \right)^n.$ But since the sum of the $t(i)$ is equal to $s,$ and $e_i/n \leq 1,$ this can be bounded from above by $(s+1)^n.$
\par Meanwhile, Proposition \ref{prop: expression for cohomology in terms of Eulerian numbers}, we have that \[|\chi(X, \mathcal{O}_X)| = \frac{1}{n!} \sum\limits_{\epsilon \in F(r, n), |\epsilon| = n} \binom{n}{\epsilon(1), \epsilon(2), \ldots, \epsilon(r)} H_1^{\epsilon(1)} H_2^{\epsilon(2)}\cdots H _r^{\epsilon(r)}.\] However, we know from intersection theory (say, \cite[Theorem 1.6.1]{lazarsfeld}) that $H_1^{\epsilon(1)} H_2^{\epsilon(2)}\cdots H _r^{\epsilon(r)} \geq \prod\limits_{i=1}^r (H_i^n)^{\epsilon(i)/n}.$ This expression on the right is at least $n!$ by Riemann-Roch (again, see \cite[\S 16]{mumford}). Therefore, $\frac{\chi(X, \mathcal{O}_X) - 1}{2}$ is bounded below by $\frac{E_r(n, 0) - 1}{2},$ and so by Lemma \ref{lem: more useable asymptotic}, we see that this is bounded below by $\frac{1}{2}\left(0.6 r^n - 1\right),$ for $n \geq 3r^2 + 100.$ 
\par Thus, since $(s+1)^n$ is increasing in $n,$ it suffices to show that \[(r-1)^n \leq 0.3r^n - \frac{1}{2};\] but we note that for $n \geq 3r^2 + 100$ that $\frac{1}{2} \leq 0.01r^n,$ so it suffices to show that $$(r-1)^n \leq 0.29r^n,$$ or that $\left(1 + \frac{1}{r-1}\right)^n \geq \frac{1}{0.29} \geq 4.$ However, note that $\left(1 + \frac{1}{r-1}\right)^{3r^2 + 100} \geq (1 + \frac{1}{r-1})^{3(r-1)^2} \geq 2^{3(r-1)} \geq 4,$ since $r \geq 2$ and we know that $(1 + \frac{1}{r-1})^{r-1}$ increases as $r$ increases. This is enough to prove the lemma.  
\end{proof}
\begin{lemma}\label{lem: large s means one term in 3s-2 region is too big}
In the same setting as the previous lemma, we have that \[\chi(X, \Omega_{X/K}^{3s-2}) < \frac{\chi(X, \Omega_{X/K}^{s-1})(\chi(X, \Omega_{X/K}^{s-1})-1)}{4}\] if $n \geq 10r^4 + 1000.$ Furthermore, if $X$ arises as the intersection of $r$ hypersurfaces that represent the same Neron-Severi class $H,$ then the inequality holds for the range $n \geq 10r^2 + 1000.$
\end{lemma}
\begin{proof}
We first handle the self-intersection case, where we know from the discussion after Proposition \ref{prop: expression for cohomology in terms of Eulerian numbers} that $\chi(X, \Omega_{X/K}^q) = (H^n)/n! E_r(n, q).$ Hence, it suffices to prove that \[E_r(n, 3s-2) < \frac{E_r(n, s-1)(E_r(n, s-1) - 1)}{4}.\] Just as in \cite[Lemma C.4]{ls25}, we split this up into two cases.
\par First, suppose that $s \leq \frac{n}{\ln (n+1) + 1} - r.$ By Lemma \ref{lem: more useable asymptotic}, and since $0.6\binom{s + r - 1}{r-1} > 1$ for $s \geq r - 1 \geq 1,$ it suffices to prove that \[\binom{3s - 2 + r - 1}{r-1}(3s - 2 + r)^n \leq 0.6^2\binom{s + r - 2}{r-1}^2 \frac{(s-1+r)^n\left((s + r - 1)^n - 1\right)}{4}.\] We start by showing that $$\binom{3s - 2 + r - 1}{r-1} \leq \binom{s + r - 1}{r-1}^2$$ for $s \geq r-1.$ This is equivalent to showing that $\prod\limits_{i=1}^{r-1} (3s - 2 + i)i \leq \prod\limits_{i=1}^{r-1} (s + i)^2.$ Note, however, that this will follow once we show that $(s+i)^2 \geq i(3s - 2 + i),$ or that $s^2 \geq (s-2)i.$ But this follows since $i \leq r - 1 \leq s.$
\par We will now show that $\frac{(s-1+r)^{2n}}{(3s - 2 + r)^n } \geq \frac{200}{9}$ for $n \geq 3r^2 + 100,$ and $s \geq r-1.$ Note that this quotient is equal to $\frac{(s+r-1)^2}{3s - 2 + r}$ raised to the $n$th power. Meanwhile, observe that $$(s + r - 1)^2 - (4/3)(3s - 2 + r) = s^2 + (2r - 6)s + r^2 - (10/3)r + (11/3) = (s + r - 3)^2 + (8/3)r - 16/3 \geq 0,$$ since $r \geq 2.$ But then we see that $$\frac{(s-1+r)^{2n}}{(3s - 2 + r)^n } = (4/3)^n \geq \frac{200}{9}$$ for $n \geq 20,$ and so in particular for $n \geq 3r^2 + 100.$
\par Now, suppose that $s > \frac{n}{\ln (n+1) + 1} - r,$ so therefore $s + r - 1 \geq \frac{n}{\ln (n+1) + 1} - 1.$ If $s \geq (n-r)/2,$ observe that $3s - 2 \geq 3(n-r)/2 - 2 > (n-r)$ (for $n \geq 3r^2 + 100$), so therefore $\chi(X, \Omega^{3s-2}_{X/K}) = 0,$ and the inequality we want to check is satisfied. Otherwise, note from Subsection \ref{subsection: generalized eulerian} that $E_r(n, q) \leq E_r(n, q+1)$ for $q < \frac{n-r}{2}.$ 
\par Note that for $n \geq 100,$ we know that $\ln (n+1) + 2 < \sqrt{n}.$ Therefore, $\sqrt{n} - 1 > \ln (n+1) + 1,$ or that $\sqrt{n} + 1 < \frac{n - 1}{\ln (n+1) + 1)} < \frac{n}{\ln(n+1) + 1}.$ Hence, if $s > \frac{n}{\ln (n+1) + 1} - r,$ then $s > \sqrt{n} - r + 1.$ Now, take some positive integer $s_0$ that lies between $\sqrt{n} - r + 1$ and $\frac{n}{\ln (n+1) + 1} - r.$ Then note that $$\frac{E_r(n, s-1)(E_r(n, s-1) - 1)}{4} \geq \frac{E_r(n, s_0 - 1)^2}{8} \geq \frac{9}{200} \binom{s_0 + r - 2}{r-1}^2(s_0 + r - 1)^{2n} \geq n^n \frac{9}{200} \binom{\sqrt{n} - 1}{{r-1}}^2.$$
\par On the other hand, note that $E_r(n, 3s - 2) \leq n! \binom{n - 1}{r-1},$ since the right-hand side is just the total number of permutations and choice of $r-1$ consecutive positions to ignore. It suffices to show that $$n^n \frac{9}{200} \binom{\sqrt{n} - 1}{{r-1}}^2 \geq \binom{n - 1}{r-1} n!.$$ 
\par We start by considering \[ \frac{\binom{\sqrt{n} - 1}{r-1}^2}{\binom{n-1}{r-1}} = \prod\limits_{i=1}^{r-1} \frac{(\sqrt{n} - i)^2}{(n-i)(r-i)}.\] But for $n \geq 3r^2 + 100,$ we have that $r - 1 \leq \frac{\sqrt{n}}{3},$ so this fraction is bounded below by $\frac{4^{r-1}}{(9(r-1))^{r-1}}.$
\par Meanwhile, by AM-GM, we can check that \[\frac{200}{9} \frac{n!}{n^n} \leq \frac{200}{9}(\frac{n+1}{2n})^n \leq \frac{200}{9} (\frac{101}{200})^n \leq \frac{200}{9} \left(\frac{101}{200}\right)^{3r^2 + 100} \leq \left(\frac{101}{200}\right)^{3r^2}.\]
\par Finally, we check that $$\left(\frac{200}{101}\right)^{3r^2} \geq \left(\frac{9r}{4}\right)^r \geq \left(\frac{9(r-1)}{4}\right)^{r-1},$$ as we can check that $\left(\frac{200}{101}\right)^{3r} \geq \frac{9r}{4}$ for $r \geq 2.$ Combining all of these inequalities is enough to prove the lemma for the self-intersection case.
\par In general, we look at $\frac{E_r(n, s-1) - 1}{4} |\chi(X, \Omega^{s-1}_{X/K})| - |\chi(X, \Omega_{X/K}^{3s-2})|$ and expand it out using the coefficients we found from Proposition \ref{prop: expression for cohomology in terms of Eulerian numbers}. It suffices to show that the coefficient of $(H_1^{e_1} H_2^{e_2} \cdots H_r^{e_r})$ is nonnegative for each tuple of the $e_i,$ since  since we know that $\chi(X,\Omega^{s-1}_{X/K}) \geq E_r(n, s-1)$ (as we argued above). That is, we want to show that \[\frac{E_r(n, s-1) - 1}{4} \sum\limits_{t \in F(r, s-1), |t| = s-1} \prod\limits_{i=1}^r E(e_i, t(i)) \geq \sum\limits_{t \in F(r, 3s-2), |t| = 3s-2} \prod\limits_{i=1}^r E(e_i, t(i)).\] As we know that $E_r(n, s-1) \geq 2,$ it suffices to prove this inequality with $\frac{E_r(n, s-1) - 1}{4}$ replaced by $E_r(n, s-1)/8.$
\par To prove that this inequality holds, we split this up into cases, just as in \cite[Lemma C.4]{ls25} and in the self-intersection case. Again, if $s \geq (n+2)/3$ this is automatic: for a term on the right-hand side to be nonzero, we require $e_i > t(i)$ for each $t,$ or that $n > |t| = 3s - 2.$
\par First, suppose that $s \leq \frac{n}{\ln (n+1) + 1} - r.$ Let $E = \max\limits_{i=1, 2, \ldots, r} e_i.$ Then, observe that \[\sum\limits_{\substack{t \in F(r, 3s - 2) \\ |t| = 3s - 2}} \prod\limits_{i=1}^r E(e_i, t(i)) \leq \sum\limits_{\substack{t \in F(r, 3s - 2) \\ |t| = 3s - 2}} \prod\limits_{i=1}^r (t(i) + 1)^{e_i} \leq \sum\limits_{\substack{t \in F(r, 3s - 2) \\ |t| = 3s - 2}} \left(\frac{\sum\limits_{i=1}^r e_it(i) + n}{n}\right)^{n}.\]
\par But note that this summation is at least $\binom{3s - 3 + r}{r-1}\left(\frac{E}{n}(3s-2) + 1\right)^n,$ since $e_i/n \leq E/n$ and the sum of the $t(i)$ is assumed to be $3s - 2.$ 
\par Next, we know by Lemma \ref{lem: more useable asymptotic} and from earlier arguments that $\chi(X, \Omega^{s-1}_{X/K}) \geq E_r(n, s-1) \geq 0.6 \binom{s + r - 2}{r-1}(s + r - 1)^n.$
\par We now consider analogous bounds for $$\sum\limits_{t \in F(r, s-1), |t| = s-1} \prod\limits_{i=1}^r E(e_i, t(i)).$$ Here, as $s \leq \frac{n}{\ln (n+1) + 1} - r,$ we have that $\frac{sE}{n} \leq \frac{E}{\ln (E+1) + 1} - 1.$ Indeed, we have that $\frac{E}{n} \geq \frac{1}{r},$ so therefore $\frac{sE}{n} \leq \frac{E}{\ln (n+1) + 1} - \frac{rE}{n} \leq \frac{E}{\ln (E+1) + 1} - 1.$
\par However, since $E \geq n/r \geq 10r \geq 20,$ we see that $3 \leq \frac{E}{\ln (E+1) + 1} - 1$ as well. Hence, in particular, we find that, if $s \geq 4,$ then one of the terms in $\sum\limits_{\substack{t \in F(r, s - 1) \\ |t| = s - 1}} \prod\limits_{i=1}^r E(e_i, t(i))$ is at least $0.6(\max(3, \frac{sE}{n}) + 1)^E.$ We will handle the $s \leq 3$ case separately. For now, we note that when $s = 2$ we replace $\max(3, \frac{sE}{n}) + 1$ with $2,$ and $s = 3$ we replace this expression with $3,$ which we will use when we handle these cases later.
\par In the case where $s \geq 4,$ it's enough to show that \[\frac{9}{200} \binom{s + r - 2}{r-1}(s+r-1)^n (\max(3, \frac{sE}{n}) + 1)^E \geq \binom{3s - 3 + r}{r-1} (\frac{E}{n}(3s - 2) + 1)^n.\] Take $n$th roots; we first observe that $$\left(\frac{\binom{3s - 3 + r}{r-1}}{\binom{s + r - 2}{r-1}}\right)^{1/n} = \left(\prod\limits_{i=0}^{r-2} \frac{3s - 3 + r - i}{s + r - 2 - i}\right)^{1/n} \leq \left(\prod\limits_{i=0}^{r-2} \frac{3s - 1}{s}\right)^{1/n} \leq 3^{\frac{r-1}{n}} \leq 3^{\frac{1}{20}} \leq 1.1.$$ Similarly, we find that $\frac{200}{9}^{1/n} \leq \frac{200}{9}^{1/100} \leq 1.04.$
\par Therefore, it suffices to show that \[(s+r-1)(\max(3, sE/n) + 1)^{E/n} \geq 1.144 (\frac{E}{n}(3s - 2) + 1).\] First, as $\frac{E}{n} \geq 1/r$ and $r \geq 2,$ note that \[\frac{3E}{n}(s+r-1) = \frac{E}{n}(3s - 2) + \frac{E}{n}(3r - 1) \geq \frac{E}{n}(3s - 2) + 1,\] or $$\frac{\frac{E}{n}(3s - 2) + 1}{s + r - 1} \leq \frac{3E}{n}.$$ So it suffices to show that $(\max(3, sE/n) + 1)^{E/n} \geq 1.144(\frac{3E}{n}) = 3.432E/n.$ 
\par To conclude, we will show that $4^x > 4^{0.9}x > 3.432x$ for $x \in (0, 1):$ taking $\log_4,$ it suffices to show $x > \log_4 x + 0.9.$ But this follows since $\log_4 x + 0.9$ is concave, so taking the tangent line at $x = \frac{1}{\ln 4}$ yields the line $x + 0.9 - \frac{1}{\ln 4} + \log_4 \frac{1}{\ln 4}.$ This constant, however is negative, so therefore $x > x + 0.9 - \frac{1}{\ln 4} + \log_4 \frac{1}{\ln 4} > \log_4 x + 0.9.$ This is enough to prove the claim when $s \geq 4.$
\par Suppose now that $s > \frac{n}{\ln (n+1) + 1} - r.$ Then, in particular, note that if we take $t_i = \frac{e_i}{\ln(n+1) + 1} - 1,$ we can increase the $t_i$ to $t_i',$ such that $t_i \leq t_i' < \frac{e_i}{2},$ such that $\sum\limits_{i=1}^r t_i' = s-1,$ since $s-1$ is assumed to be at most $\frac{n-1}{3} < n/2 - r.$ Furthermore, we know that $E(n, r)$ is unimodal and symmetric (being log-concave and symmetric).
\par Thus, \[\sum\limits_{t \in F(r, s-1), |t| = s-1} \prod\limits_{i=1}^r E(e_i, t(i)) \geq \prod\limits_{i=1}^r E(e_i, \frac{e_i}{\ln(n+1) + 1} - 1),\] which by Lemma \ref{lem: more useable asymptotic} is bounded below by \[0.6^r \prod\limits_{i=1}^r \left(\frac{e_i}{\ln(n+1) + 1}\right)^{e_i}.\] On the other hand, observe that \[\sum\limits_{t \in F(r, s-1), |t| = 3s-2} \prod\limits_{i=1}^r E(e_i, t(i)) \leq \sum\limits_{t \in F(r, e_i)} \prod\limits_{i=1}^r E(e_i, t(i)) = \prod\limits_{i=1}^r e_i!.\]
\par Hence, to prove the desired inequality, it suffices to show that $\frac{E_r(n, s-1)}{8} \geq \left(\frac{5}{3}\left(\ln(n+1) + 1\right)\right)^n.$ Since $E_r(n, s-1)$ is log-concave and symmetric, the sequence with $s = 1, 2, \ldots, n/2$ is increasing. Taking $s_0$ to be the largest integer smaller than $\frac{n}{\ln (n+1) + 1} - r,$ we see that \[\frac{E_r(n, s-1)}{8} \geq \frac{E_r(n, s_0)}{8} \geq \frac{3}{40} \binom{s_0 + r - 1}{r-1} (s_0 + r)^n\] by Lemma \ref{lem: more useable asymptotic}. So it suffices to prove that $\frac{3}{40} \binom{s_0 + r - 1}{r-1} \geq 1$ and $s_0 + r \geq \frac{5}{3}\left(\ln(n+1) + 1\right).$ The first, since $r \geq 2,$ holds when $s_0 \geq 14.$ The second, meanwhile, holds when $s_0 \geq \frac{5}{3} \left(\ln(n+1) + 1\right).$ 
\par Thus, we want to show that \[\frac{n}{\ln (n+1) + 1} - r > 14, \frac{5}{3} \left(\ln(n+1) + 1\right).\] For the first, we already know that $\ln(n+1) + 2 < \sqrt{n}$ for $n \geq 100,$ so it suffices to show that $\sqrt{n} > r + 14;$ this holds when $n \geq r^2 + 28r + 196,$ so in particular when $n \geq 10r^4 + 1000$). Meanwhile to show that $\frac{n}{\ln (n+1) + 1} - r > \frac{5}{3} \ln(n+1) + 1,$ first observe that $\frac{n}{\ln (n+1) + 1} > \sqrt{n},$ and the derivative of $\sqrt{n}$ is $\frac{1}{2\sqrt{n}}.$ The derivative of the right-hand side, meanwhile, is $\frac{1}{n+1},$ so the left-hand side grows faster when $n \geq 9.$ It suffices to find, for each $r,$ one $n$ for which $\sqrt{n} - r > \frac{5}{3} \ln(n+1) + 1.$
\par We take $n = 10r^4 + 1000;$ in this case, note that $\frac{5}{3}\ln (n+1) + 1 \leq \frac{20}{3}\ln (10r+10) + 1.$ Taking derivatives on both sides, we get $\frac{20r^3}{\sqrt{10r^4 + 1000}} - 1$ on the left and $\frac{20}{3(r+1)};$ the former is larger for $r \geq 2.$ Finally, note that plugging in $r = 2,$ we see that $\sqrt{10r^4 + 1000} - r > 29 > \frac{5}{3} \ln 1161 + 1,$ which is enough.
\par Finally, suppose that $s \leq 3,$ so hence $r = 2, 3,$ or $4.$ We do these cases manually.
\par Suppose that $s = 2:$ then, note that $\frac{E}{n}(3s - 2) + 1 = 4E/n + 1,$ and $s + r - 1 \geq 3,$ so therefore the inequality we want to prove from the first case is $2^{E/n} > 1.144(4E/3n + 1/3);$ it suffices to show that $2^x > 1.55x + 0.4$ for $x \in (0, 1).$ But this follows since the derivative of $2^x$ is increasing in $x,$ equals $2 \ln 2 < 1.55$ at $x = 1,$ and $2^1 = 2 > 1.55\cdot 1 + 0.4.$
\par Meanwhile, for $s = 3,$ we have $\frac{E}{n}(3s - 2) + 1 = \frac{7E}{n} + 1,$ and $s + r - 1 \geq 4,$ so we want to show $3^{E/n} > 1.144(7E/4n + 1/4).$ Note that $3^{0.6}\ln 3 > 2.002 = \frac{7\cdot 1.144}{4},$ so it suffices to show that $3^x > (3^{0.6} \ln 3)x + 0.4.$ But note that this is parallel to the tangent line at $x = 0.6,$ and furthermore $3^{0.6} > 1.93 > 3^{0.6} \ln 3 \cdot 0.6 + 0.4.$ Hence, since $3^x$ is a convex function, we see that $3^x > (3^{0.6} \ln 3)x + 0.4$ for $x \in (0, 1).$ This handles the case for $s = 3.$
\par Finally suppose that $s = 1,$ so we must have $r = 2.$ We have to handle this case differently from the others. In this case, we have that \[\chi(X, \Omega^{s-1}_X) = \chi(X, \mathcal{O}_X) = \sum\limits_{i=1}^{n-1} \frac{1}{i!(n-i)!}H_1^iH_2^{n-i},\] and \[\chi(X, \Omega^{3s-2}_X) = \chi(X, \Omega_X) = \sum\limits_{i=1}^{n-1} \frac{E(i, 1) + E(n-i, 1)}{i!(n-i)!}H_1^iH_2^{n-i}.\] Just as in the previous cases, we will show that \[\chi(X, \Omega_{X/K}^1) < \frac{\chi(X, \mathcal{O}_X)^2}{8}.\] We know from Lemma \ref{lem: Eulerian asymptotic} that $E(i, 1) \leq 2^i,$ so it suffices to show that \[\left(\sum\limits_{i=1}^{n-1} \frac{1}{i!(n-i)!}H_1^iH_2^{n-i}\right)^2 \geq 8 \sum\limits_{i=1}^{n-1} \frac{2^i + 2^{n-i}}{i!(n-i)!}H_1^iH_2^{n-i}.\]
\par To prove this, we will use some facts about intersection numbers of hypersurfaces from \cite{lazarsfeld}. First, observe that from \cite[Example 1.6.4]{lazarsfeld} that the sequence given by $(H_1^aH_2^{n-a}),$ as $a$ goes from $1, 2, \ldots, n-1,$ is log concave. In particular, if $1 \leq a < b \leq \frac{n}{2},$ we have that \[(H_1^aH_2^{n-a})(H_1^{n-a}H_2^a) \leq (H_1^bH_2^{n-b})(H_1^{n-b}H_2^b).\]
\par Hence, $$\sum\limits_{i=1}^{n-1} \frac{1}{i!(n-i)!}H_1^iH_2^{n-i} = \sum\limits_{i=1}^{n-1} \frac{1}{i!(n-i)!}\frac{H_1^iH_2^{n-i} + H_1^{n-i}H_2^i}{2}$$ $$\geq \frac{1}{n!} \sum\limits_{i=1}^{n-1} \binom{n}{i} \sqrt{H_1^iH_2^{n-i} H_1^{n-i}H_2^i} \geq \frac{\sqrt{H_1^4H_2^{n-4} H_1^{n-4}H_2^4}}{n!} \sum\limits_{i=4}^{n-4} \binom{n}{i} = \frac{\sqrt{H_1^4H_2^{n-4} H_1^{n-4}H_2^4}}{n!} \left(2^n - \frac{n^3 + 5n + 6}{3}\right).$$ Furthermore, observe that $$\frac{\sum\limits_{i=1}^{n-1} \frac{2^{n-i} + 2^i}{i!(n-i)!}H_1^iH_2^{n-i}}{\sum\limits_{i=1}^{n-1} \frac{1}{i!(n-i)!}H_1^iH_2^{n-i}} \leq 2^{n-1} + 2,$$ since $2^x$ is convex, so $2^{n-1} + 2 \geq 2^{n-i} + 2^i$ for each $i \in [1, n-1].$ In particular, if $\frac{\sqrt{H_1^4H_2^{n-4} H_1^{n-4}H_2^4}}{n!} \geq 8\frac{2^{n-1} + 2}{2^n - \frac{n^3 + 5n + 6}{3}}$ then we have our desired inequality, since then we will have that \[\sum\limits_{i=1}^{n-1} \frac{1}{i!(n-i)!}H_1^iH_2^{n-i} \geq 8 \frac{\sum\limits_{i=1}^{n-1} \frac{2^i + 2^{n-i}}{i!(n-i)!}H_1^iH_2^{n-i}}{\sum\limits_{i=1}^{n-1} \frac{1}{i!(n-i)!}H_1^iH_2^{n-i}}.\]
\par Otherwise, suppose this is not the case. Rewrite this bound as \[8\frac{2^{n-1} + 2}{2^n - \frac{n^3 + 5n + 6}{3}} = 4 + \frac{\frac{4n^3 + 20n}{3} + 24}{2^n - \frac{n^3 + 5n + 6}{3}}.\] We first claim that for $n \geq 100,$ \[2^n - \frac{n^3 + 5n + 6}{3} \geq 50(\frac{4n^3 + 20n}{3} + 24);\] that is, \[2^n \geq 67n^3 + 335n + 1202.\] This can be checked by observing that this inequality holds when $n = 100,$ then noting that \[\frac{67(n+1)^3 + 335(n+1) + 1202}{67n^3 + 335n + 1202} < 2\] for $n \geq 100.$ Hence, we find that \[8\frac{2^{n-1} + 2}{2^n - \frac{n^3 + 5n + 6}{3}} \leq 4.02 < \sqrt{17}.\] 
\par Furthermore, from Riemann-Roch, we know that $H_1^n, H_2^n$ are at least $n!,$ and $H_1^{n-4}H_2^4 \geq (H_1^n)^{(n-4)/n}(H_2^n)^{4/n} \geq n!.$ Thus, both $H_1^{n-4}H_2^4$ and $H_1^{n-4}H_2^4$ are at least $n!$ and multiply to less than $17(n!)^2,$ which means $\frac{H_1^{n-4}H_2^4}{H_1^4H_2^{n-4}}$ must lie between $\frac{1}{17}$ and $17.$ Since $H_1^{n-i}H_2^i H_1^i H_2^{n-i} \leq H_1^{n-4}H_2^4 H_1^4H_2^{n-4}$ for $i \in \{1, 2, 3\},$ we see more generally that $\frac{H_1^{n-i}H_2^i}{H_1^iH_2^{n-i}} \in [1/17, 17]$ for $i \in \{1, 2, 3, 4\}.$
\par From this, we rewrite our inequality \[\left(\sum\limits_{i=1}^{n-1} \frac{1}{i!(n-i)!}H_1^iH_2^{n-i}\right)^2 \geq 8 \sum\limits_{i=1}^{n-1} \frac{2^i + 2^{n-i}}{i!(n-i)!}H_1^iH_2^{n-i}\] as \[\sum\limits_{i=1}^{n-1} \left(- (2^{i+3} + 2^{n-i+3}) + \sum\limits_{i=1}^{n-1} \frac{1}{i!(n-i)!}H_1^iH_2^{n-i}\right) \frac{H_1^iH_2^{n-i} + H_1^{n-i}H_2^i}{2} \frac{1}{i!(n-i)!} \geq 0.\] To tackle this inequality, first recall from previous arguments in this section and Lemma \ref{lem: Eulerian asymptotic} that \[\sum\limits_{i=1}^{n-1} \frac{1}{i!(n-i)!}H_1^iH_2^{n-i} \geq E_2(n, 0) \geq 2^n - (n+1),\] so in particular notice that $- (2^{i+3} + 2^{n-i+3}) + \sum\limits_{i=1}^{n-1} \frac{1}{i!(n-i)!}H_1^iH_2^{n-i} \geq 0$ if $i \geq 4$ and $i \leq n-4.$ Thus, the only negative terms that could occur are at $i = 1, 2, 3$ (and the corresponding terms by symmetry). We will use the significantly stronger bound that, as $n \geq 10r^2 + 1000 > 1000,$ $2^n - (n+1) \geq 0.99 \cdot 2^n$ (one can check that $n+1 \leq 0.01 \cdot 2^n$ for $n \geq 11$).
\par Furthermore, for each $j \geq 5$ and $i \in \{1, 2, 3, 4\},$ by log concavity we know that $H_1^jH_2^{n-j} \geq (H_1^iH_2^{n-i})^{\frac{n - j - i}{n-2i}} (H_1^{n-i} H_2^i)^{\frac{j - i}{n-2i}}.$ Therefore, it follows that \[H_1^jH_2^{n-j} + H_1^{n-j}H_2^j \geq (H_1^iH_2^{n-i})^{\frac{n - j - i}{n-2i}} (H_1^{n-i} H_2^i)^{\frac{j - i}{n-2i}} + (H_1^iH_2^{n-i})^{\frac{j - i}{n-2i}} (H_1^{n-i} H_2^i)^{\frac{n - j - i}{n-2i}}.\] Letting $M_i$ be the maximum of $H_1^{n-i} H_2^i, H_1^i H_2^{n-i},$ $m_i$ the minimum, we obtain \[M_i^{\frac{n - j - i}{n-2i}} m_i^{\frac{j - i}{n-2i}} + m_i^{\frac{j - i}{n-2i}} M_i^{\frac{n - j - i}{n-2i}} \geq (M_i + m_i) \left(\frac{m_i}{M_i}\right)^{\min(\frac{j-i}{n-2i}, \frac{n-j-i}{n-2i})}.\] But we have argued previously that $\frac{m_i}{M_i} \geq \frac{1}{17}.$ In particular, if we let $M$ be the maximum of $H_1^iH_2^{n-i} + H_1^{n-i}H_2^i$ for $i = 1, 2, 3,$ we find that, for $j \leq n/2,$ \[(H_1^iH_2^{n-i})^{\frac{n - j - i}{n-2i}} (H_1^{n-i} H_2^i)^{\frac{j - i}{n-2i}} \geq \frac{1}{17}^{\frac{j-1}{n-6}}M \geq \frac{1}{17}^{\frac{j-1}{100}}M,\] as we are assuming $n \geq 10r^2 + 1000 > 100.$ Similarly, if $j \geq n/2$ this instead becomes $\frac{1}{17}^{\frac{n-j-1}{100}}M$
\par Hence, analyzing \[\sum\limits_{i=1}^{n-1} \left(- (2^{i+3} + 2^{n-i+3}) + \sum\limits_{i=1}^{n-1} \frac{1}{i!(n-i)!}H_1^iH_2^{n-i}\right) \frac{H_1^iH_2^{n-i} + H_1^{n-i}H_2^i}{2} \frac{1}{i!(n-i)!},\] we see that the magnitude of the negative terms on the left-hand side are bounded above by $M(2^{n+2} + 2^{n+1} + 2^n + 8 + 16 + 32) \leq 8M\cdot 2^n.$ On the other hand, we see that taking $j$ between $4$ and $43$ (as well as the same range for $n-j$), and recalling that \[\sum\limits_{i=1}^{n-1} \frac{1}{i!(n-i)!}H_1^iH_2^{n-i} \geq 0.99 \cdot 2^n,\] we obtain terms bounded below by $$\frac{1}{17}^{0.42} M\left(40 \cdot 0.99\cdot2^n -  \left(\sum\limits_{j=4}^{40} 2^{n-j+3} + 2^{j + 3}\right)\right) \geq M(12 \cdot 2^n - 2^n - 2^{44}) \geq 10M\cdot 2^n,$$ which is larger. In particular, this means that \[\sum\limits_{i=1}^{n-1} \left(- (2^{i+3} + 2^{n-i+3}) + \sum\limits_{i=1}^{n-1} \frac{1}{i!(n-i)!}H_1^iH_2^{n-i}\right) \frac{H_1^iH_2^{n-i} + H_1^{n-i}H_2^i}{2} \frac{1}{i!(n-i)!} \geq 0,\] which is what we wanted to show.
\end{proof}
\subsection{Other Inequalities}
For the inequalities in this subsection, we will first prove the inequality for the case where $X$ is given by the intersection of $r$ hypersurfaces all representing the Neron-Severi class $H,$ before going to the slightly more general setting. This will let us obtain the wider ranges in this first self-intersection case. In addition, the spirit of the argument is captured by the first with fewer technical obstructions.
\begin{lemma}\label{lem: the respective ratios are big enough}
For the ranges of $n$ given in Theorem \ref{thm: big monodromy statement}, we have that $\frac{|\chi(X, \Omega^1_{X/K})|}{|\chi(X, \mathcal{O}_X)|} > 16.$
\end{lemma}
\begin{proof}
We compare coefficients and show that $|\chi(X, \Omega^1_{X/K})| - 16 |\chi(X, \mathcal{O}_X)| > 0.$ To show this, expanding these out using Proposition \ref{prop: expression for cohomology in terms of Eulerian numbers}, we see that the coefficient of $(H_1^{e_1}\cdots H_r^{e_R})$ is given by \[\frac{1}{n!} \binom{n}{e_1, e_2,\ldots, e_r} \left(-16 + \sum\limits_{i=1}^r E(e_i, 1) \right);\] it thus suffices to show for each choice of $e_i$ that $-16 + \sum\limits_{i=1}^r E(e_i, 1)  > 0.$ But to check this, since the sum of the $e_i$ is $n,$ one of the $e_i$ is at least $10r.$ But we know from Lemma \ref{lem: Eulerian asymptotic} that $E(e_i, 1) \geq 2^{e_i} - (e_i+1),$ which for $e_i \geq 10r \geq 6$ is at least $64 - 7 = 57 > 16,$ so we are done.
\end{proof}
\begin{lemma}\label{lem: the core q=2 inequality}
For the ranges of $n$ given in Theorem \ref{thm: big monodromy statement}, we have \[\frac{1}{48} \left(\frac{\chi(X, \Omega^1_{X/K})}{\chi(X, \mathcal{O}_X)}\right)^2 > \frac{\chi(X, \Omega^2_{X/K})}{\chi(X, \mathcal{O}_X)}.\]
\end{lemma}
\begin{proof}
We start with the case that $X$ is the intersection of $r$ representatives of the same Neron-Severi class. To start, begin by employing Lemma \ref{lem: more useable asymptotic}. We then have that $\chi(X, \Omega^1_{X/K})/\chi(X, \mathcal{O}_X) > 0.6 r \left( 1 + \frac{1}{r}\right)^n$ and $\chi(X, \Omega^2_{X/K})/\chi(X, \mathcal{O}_X) < \frac{5}{3} \frac{(r+1)r}{2} \left(1 + \frac{2}{r} \right)^n.$ 
\par In other words, it thus suffices to show that \[\frac{1}{48} r^2 \left( 1 + \frac{1}{r}\right)^{2n} > \frac{125}{27} \frac{(r+1)r}{2} \left(1 + \frac{2}{r}\right)^n,\] or that \[\frac{r}{(r+1)} \left(1 + \frac{1}{r^2 + 2r}\right)^n > \frac{1000}{9}.\] Finally, using the fact that we have $r \geq 2,$ it suffices to show that $\left(1 + \frac{1}{r^2 + 2r}\right)^n > \frac{500}{3}.$
\par Now, we know that $r^2 + 2r > 6,$ and that $(1+\frac{1}{x})^x$ is increasing in positive $x;$ therefore, we can bound the left-hand side from below as $2^{\frac{n}{r^2 + 2r}}.$ But with $n > 10r^2 + 1000,$ we observe that $\frac{10r^2 + 1000}{r^2 + 2r} > 9,$ as $r^2 - 18r + 1000 > 0.$ This proves the desired inequality in the self-intersection case, since $2^9 > \frac{500}{3}.$
\par Now, suppose we are in the case where our hypersurfaces represent Neron-Severi classes $d_1H, d_2H, \ldots, d_rH.$ Without loss of generality, suppose that $d_1 \leq d_2 \leq \cdots \leq d_r.$ We start by building a lower bound for $\chi(X, \Omega^1_{X/K})$ and upper bounds for $\chi(X, \mathcal{O}_X), \chi(X, \Omega^2_{X/K}).$
\par First, for the upper bounds, we see that \begin{align*}|\chi(X, \mathcal{O}_X)| & = \frac{H^n}{n!} \sum\limits_{e \in F(r, n), |e| = n} \binom{n}{e(1), e(2), \ldots, e(r)}d_1^{e(1)} \cdots d_r^{e(r)} & = H^n \sum\limits_{e \in F(r, n), |e| = n} \prod\limits_{i=1}^r \frac{d_i^{e(i)}}{e(i)!} \\ &  = H^n \sum\limits_{e \in F(r, n), |e| = n} \prod\limits_{i=1}^r \frac{d_i}{e(i) }\frac{d_i^{e(i) - 1}}{(e(i) - 1)!} & \leq H^n \sum\limits_{e \in F(r, n), |e| = n} \prod\limits_{i=1}^r d_i \frac{d_i^{e(i) - 1}}{(e(i) - 1)!} \\ & = \frac{H^n}{(n-r)!} d_1d_2 \cdots d_r (d_1 + d_2 + \cdots + d_r)^{n-r},\end{align*} and \begin{align*}|\chi(X, \Omega^2_{X/K})| & = \frac{H^n}{n!} \sum\limits_{e \in F(r, n), |e| = n} \left(\sum\limits_{i=1}^r E(e(i), 2) + \sum\limits_{1 \leq i < j \leq r} E(e(i), 1)E(e(j), 1)\right)\binom{n}{e(1), e(2), \ldots, e(r)}d_1^{e(1)} \cdots d_r^{e(r)} \\ & = H^n \sum\limits_{e \in F(r, n), |e| = n} \left(\sum\limits_{i=1}^r E(e(i), 2) + \sum\limits_{1 \leq i < j \leq r} E(e(i), 1)E(e(j), 1)\right)\prod\limits_{i=1}^r \frac{d_i^{e(i)}}{e(i)!} \\ & \leq H^n \sum\limits_{e \in F(r, n), |e| = n} \left(\sum\limits_{i=1}^r 3^{e(i)} + \sum\limits_{1 \leq i < j \leq r} 2^{e(i) + e(j)}\right)\prod\limits_{i=1}^r \frac{d_i}{e(i) }\frac{d_i^{e(i) - 1}}{(e(i) - 1)!} \\ & \leq H^n \sum\limits_{e \in F(r, n), |e| = n} \left(\sum\limits_{i=1}^r 3^{e(i)} + \sum\limits_{1 \leq i < j \leq r} 2^{e(i) + e(j)}\right)\prod\limits_{i=1}^r d_i \frac{d_i^{e(i) - 1}}{(e(i) - 1)!} \\ & = \frac{H^n}{(n-r)!}d_1d_2 \cdots d_r \left(3\sum\limits_{i=1}^r (d_1 + d_2 + \cdots + d_r + 2d_i)^{n-r} + 4\sum\limits_{1 \leq i < j \leq r} (d_1 + d_2 + \cdots + d_r + d_i + d_j)^{n-r}\right).\end{align*}
\par Meanwhile, for the lower bound, we first claim that $$\sum\limits_{i=1}^r E(e(i), 1) \geq 0.6 \cdot 2^{e(r)}.$$ To see this, we use Lemma \ref{lem: Eulerian asymptotic}. Indeed, we know that $E(n, 1) \geq 2^n - (n+1)$ for all $n.$ If $e(r) \geq 4,$ we get the desired bound. Otherwise, $e(r) < 4;$ but then there is some $i$ for which $e(i) \geq 4,$ assuming $n > 10r^4 + 1000 \geq 4r,$ and so $0.6 \cdot 2^{e(r)} \leq 0.6 \cdot 2^{e(i)} \leq E(e(i), 1) \leq \sum\limits_{i=1}^r E(e(i), 1).$ Using this, along with the AM-GM inequality, we can then obtain the lower bound \begin{align*} |\chi(X, \Omega^1_{X/K})| & = \frac{H^n}{n!} \sum\limits_{e \in F(r, n), |e| = n} \left(\sum\limits_{i=1}^r E(e(i), 1)\right)\binom{n}{e(1), e(2), \ldots, e(r)}d_1^{e(1)} \cdots d_r^{e(r)} \\ & = H^n \sum\limits_{e \in F(r, n), |e| = n} \left(\sum\limits_{i=1}^r E(e(i), 1)\right)\prod\limits_{i=1}^r \frac{d_i^{e(i)}}{e(i)!} \\ & \geq H^n \sum\limits_{e \in F(r, n), |e| = n} \left(0.6 \cdot 2^{e(r)}\right)\prod\limits_{i=1}^r \frac{d_i}{e(i) }\frac{d_i^{e(i) - 1}}{(e(i) - 1)!} \\ & \geq H^n \sum\limits_{e \in F(r, n), |e| = n} \left(0.6 \cdot 2^{e(r)}\right)\prod\limits_{i=1}^r \frac{d_i}{(n/r)} \frac{d_i^{e(i) - 1}}{(e(i) - 1)!} \\ & = \frac{H^n}{(n-r)!}\frac{1.2 r^r}{n^r} d_1d_2 \cdots d_r \left((d_1 + d_2 + \cdots + 2d_r)^{n-r}\right).
\end{align*}
\par Using the bounds above, and cancelling out like terms, it thus suffices to show that $$\left(\frac{r^r}{n^r} (d_1 + d_2 + \cdots + 2d_r)^{(n-r)}\right)^2$$ is at least \[40 (d_1 + d_2 + \cdots + d_r)^{n-r} \left(3\sum\limits_{i=1}^r (d_1 + d_2 + \cdots + d_r + 2d_i)^{n-r} + 4\sum\limits_{1 \leq i < j \leq r} (d_1 + d_2 + \cdots + d_r + d_i + d_j)^{n-r}\right);\] this expression can be bounded above by \[40 (d_1 + d_2 + \cdots + d_r)^{n-r} \left(3r (d_1 + d_2 + \cdots + 3d_r )^{n-r} + (2r^2 - 2r) (d_1 + d_2 + \cdots + 3d_r)^{n-r}\right).\]
\par To do this, let $D = d_1 + \cdots + d_r;$ note that $D \leq rd_r.$ Then, observe that $$\frac{(d_1 + d_2 + \cdots + d_{r-1} + 2d_r)^2}{(d_1 + d_2 + \cdots + d_{r-1} + d_r)(d_1 + d_2 + \cdots + d_{r-1} + 3d_r)} = \frac{D^2 + 2Dd_r + d_r^2}{D^2 + 2Dd_r} = 1 + \frac{d_r^2}{D^2 + 2Dd_r} \geq 1 + \frac{1}{r^2 + 2r}.$$
\par Thus, to prove the lemma, it suffices to show that for the ranges of $n$ that we have given, $$\left(1 + \frac{1}{r^2 + 2r}\right)^{n-r} \geq 40\frac{2r^2 + r}{r^{2r}} n^{2r}.$$
\par Again, we observe that $\left(1 + \frac{1}{r^2 + 2r}\right)^{r^2 + 2r} \geq 2,$ so it suffices to show $2^{\frac{n-r}{r^2 + 2r}} \geq 40(2r^2 + r)\left(\frac{n}{r}\right)^{2r}$ when $n \geq 10r^4 + 1000.$ But $\frac{r}{r^2 + 2r} = \frac{1}{r+2} \leq 1/4,$ so as $2^{1/4} \leq \frac{5}{4},$ it suffices to show $2^{\frac{n}{r^2 + 2r}} \geq 50(2r^2 + r)\left(\frac{n}{r}\right)^{2r}$
\par Now, it suffices to show $2^{n/(r^2 + 2r)} \geq\left(\frac{n}{r}\right)^{2r + 2},$ since $50, 2r^2 + r \leq 10r^3 < n/r.$ Furthermore, with $r \geq 2,$ it suffices to prove the inequality with $2r + 2$ replaced with $3r.$ Taking $\log_2,$ this is equivalent to showing that $\frac{n}{r^2 + 2r} \geq 3r \log_2(\frac{n}{r})$ for $n \geq 10r^4 + 1000.$ Letting $x = \frac{n}{r},$ we want to show that $$x > 3r(r+2) \log_2x$$ for $x \geq \frac{10r^4 + 1000}{r}.$ Note that the slope of the right-hand side equals $1$ when $x = \frac{3r(r+2)}{\ln 2} \leq 10r^3,$ and furthermore that $3r(r+2) \log_2 x$ is a concave function. Therefore, if we prove the inequality holds for some $\frac{3r(r+2)}{\ln 2} \leq x < \frac{10r^4 + 1000}{r},$ then it will hold for all larger $x$ (since $3r(r+2) \log_2 x$ has slope less than $1$ for $x \geq \frac{3r(r+2)}{\ln 2}$). 
\par But now at $x = 10r^3,$ we have that $\frac{10r^2}{3(r+2)} > \log_2 10r^3$ for $r \geq 3,$ by using a similar argument to the above and a bit of direct computation. For $r = 2,$ we can check that $x = \frac{10r^4 + 1000}{r}$ satisfies the desired inequality, and so we are done.
\end{proof}
\begin{lemma}\label{lem: remaining terms in q=2}
For the ranges of $n$ given above, we have 
\begin{align*}\frac{|\chi(X, \Omega^2_{X/K})|}{|\chi(X, \mathcal{O}_X)|} & > \frac{n-r}{2} \frac{|\chi(X, \Omega^1_{X/K})|}{|\chi(X, \mathcal{O}_{X})|} + \left(\sqrt{\frac{18|\chi(X, \Omega_{X/K}^5)|}{|\chi(X, \Omega_{X/K}^1)|}}\right) + \sqrt{\frac{9|\chi(X, \Omega_{X/K}^3)||\chi(X, \Omega_{X/K}^1)|}{2|\chi(X, \mathcal{O}_{X})^2|}} \\ & + \left(\frac{9|\chi(X, \Omega_{X/K}^3)|}{|\chi(X, \Omega_{X/K}^1)|}\right) + 3.\end{align*}
\end{lemma}
\begin{proof}
Again, we start with the self-intersection case. We observe that, using Lemma \ref{lem: more useable asymptotic}, noting that $r \geq 2$ and $(H^n)/n! \geq 1,$ that the left-hand side is at most $0.3 r(r+1) \left(1 + \frac{2}{r}\right)^n.$ So we will show that the second, third, and fourth terms are less than $\frac{3}{40} r(r+1)\left(1 + \frac{2}{r}\right)^n$ (a quarter of the lower bound of the left-hand side), and the first and fifth terms are less than $\frac{3}{80} r(r+1)\left(1 + \frac{2}{r}\right)^n$ 
\par For the first, we know that this term is at most $\frac{5(n-r)}{6} r (1 + \frac{1}{r})^n.$ But then note that this equals \[\left(\frac{3}{80} r(r+1)\left(1 + \frac{2}{r}\right)^n\right) \left(\frac{200(n-r)}{9(r+1)}\right) \left(\frac{r+1}{r+2}\right)^n.\] It thus suffices to show that $$\frac{200(n-r)}{9(r+1)} < \left(1 + \frac{1}{r+1}\right)^n.$$ With $r \geq 2,$ we see that $\left(\left(1 + \frac{1}{r+1}\right)^{r+1}\right)^\frac{n}{r+1},$ which again, since $(1 + \frac{1}{x})^x$ increases to $e,$ can be bounded below by $2^\frac{n}{r+1}$ (as $r \geq 2 > 1$).
\par Thus, it suffices to show that $\frac{200(n-r)}{9(r+1)} < 2^{\frac{n}{r+1}}.$ But we see that $2^x > \frac{200x}{9}$ for $x \geq 8,$ and certainly $\frac{n}{r+1} > 8$ for $n \geq 10r^2 + 1000;$ hence, the first term in the right-hand side is less than a quarter of $|\chi(X, \Omega^2)|/|\chi(X, \mathcal{O}_X)|.$
\par For the second term, using Lemma \ref{lem: more useable asymptotic}, we want to prove that \[\frac{3r(r+1)}{40} \left(1 + \frac{2}{r}\right)^n > \sqrt{18} \sqrt{\frac{(r+1)(r+2)(r+3)(r+4)}{72}} \left( 1 + \frac{4}{r+1} \right)^{n/2}.\] Rearranging this, we want to show that
\[\left( 1 + \frac{8r + 4}{r^3 + 5r^2}\right)^{n/2} > \frac{20}{3} \sqrt{\frac{(r+2)(r+3)(r+4)}{r^2(r+1)}}.\] That is, we want to show $\left( 1 + \frac{8r + 4}{r^3 + 5r^2}\right)^n > \frac{400}{9}\frac{(r+2)(r+3)(r+4)}{r^2(r+1)}.$ But we check that $\frac{8r + 4}{r^3 + 5r^2} > \frac{1}{r^2},$ and the left-hand side is bounded below by $\left(\left(1 + \frac{1}{r^2}\right)^{r^2}\right)^{n/r^2} > 2^{n/r^2} > 2^{10},$ which is enough to handle the second term since for $r \geq 2,$ we have $\frac{(r+2)(r+3)(r+4)}{r^2(r+1)} \leq \frac{4 \cdot 5 \cdot 6}{4 \cdot 3} = 10,$ so $\frac{400}{9}\frac{(r+2)(r+3)(r+4)}{r^2(r+1)} < 4000/9 < 1024.$
\par For the third term, applying Lemma \ref{lem: more useable asymptotic} again, it suffices to show \[\frac{3r(r+1)}{40} \left(1 + \frac{2}{r}\right)^n > 3 \sqrt{\frac{r^2(r+1)(r+2)}{12}} \left(1 + \frac{4r + 3}{r^2} \right)^{n/2},\] or that \[ \left(1 + \frac{1}{r^2 + 4r + 3}\right)^{n/2} > 40\sqrt{\frac{(r+2)}{12(r+1)}}.\] Note that the right-hand side, as $r \geq 2,$ is at most $40/3.$ So we want to show that $\left(1 + \frac{1}{r^2 + 4r + 3}\right)^{n} > \frac{1600}{9}.$ But again we see that the left-hand side is at least $(64/27)^{\frac{n}{r^2 + 4r + 3}}.$ Note that  $$\frac{n}{r^2 + 4r + 3} \geq \frac{10r^2 + 100}{r^2 + 4r + 3}  = 7 + \frac{3r^2 - 28 r + 73}{r^2 + 4r +3} > 7.$$ as $3r^2 - 28r + 73 > 0$ for all $r.$ Finally, as $1600/9 < 200 < (64/27)^7,$ we have bounded the third term.
\par For the fourth term, another application of Lemma \ref{lem: more useable asymptotic} shows that it's enough to show that \[\frac{3}{80} r(r+1) \left(1 + \frac{2}{r}\right)^n > \frac{5(r+1)(r+2)}{2} \left(1 + \frac{2}{r+1} \right)^n,\] which we can rearrange as proving the inequality \[\left(1 + \frac{2}{r^2 + 3r}\right)^n > \frac{200(r+2)}{3r}.\] But the right-hand side is bounded above by $\frac{400}{3}$ and the left-hand side is \[\left(\left(1 + \frac{2}{r^2 + 3r}\right)^{\frac{r^2 + 3r}{2}}\right)^{\frac{2n}{r^2 + 3r}} \geq \left(\frac{625}{256}\right)^{\frac{2n}{r^2 + 3r}} \geq 2^9 > \frac{400}{3},\] which is enough.
\par Finally, for the last term, note that $\frac{3r(r+1)}{80} \left(1 + \frac{2}{r}\right)^n > \frac{9}{40} 2^{2n/r} > 3$ for $2n/r > 5$ (which holds for $n > 10r^2 + 1000$). Combining the above five inequalities proves the lemma in the self-intersection case. 
\par We now consider the case of intersection of hypersurfaces which, in the Neron-Severi group, are represented by $d_1H, d_2H, \ldots, d_rH,$ for positive integers $d_1, d_2, \ldots, d_r.$ Again, suppose that $d_1 \leq d_2 \leq \cdots \leq d_r.$ We use the same argument as in the above lemma, by Lemma \ref{lem: the core q=2 inequality}. Using this lemma and AM-GM, we have \begin{align*}|\chi(X, \Omega^q_{X/K})| & = \frac{H^n}{n!} \sum\limits_{e \in F(r, n), |e| = n} \left(\sum\limits_{t \in F(r, q), |t| = q} \prod\limits_{i=1}^r E(e(i), t(i))\right)\binom{n}{e(1), e(2), \ldots, e(r)}d_1^{e(1)} \cdots d_r^{e(r)} \\ & = H^n \sum\limits_{e \in F(r, n), |e| = n} \left(\sum\limits_{t \in F(r, q), |t| = q} \prod\limits_{i=1}^r E(e(i), t(i))\right)\prod\limits_{i=1}^r \frac{d_i^{e(i)}}{e(i)!} \\ & = H^n \sum\limits_{e \in F(r, n), |e| = n} \left(\sum\limits_{t \in F(r, q), |t| = q} \prod\limits_{i=1}^r E(e(i), t(i))\right)\prod\limits_{i=1}^r \frac{d_i}{e(i) }\frac{d_i^{e(i) - 1}}{(e(i) - 1)!} \\ & \leq H^n \sum\limits_{e \in F(r, n), |e| = n} \left(\sum\limits_{t \in F(r, q), |t| = q} \prod\limits_{i=1}^r (t(i) + 1)^{e(i)}\right)\prod\limits_{i=1}^r d_i \frac{d_i^{e(i) - 1}}{(e(i) - 1)!} \\ & = \frac{H^n}{(n-r)!}d_1d_2 \cdots d_r \left(\sum\limits_{t \in F(r, q), |t| = q} \prod\limits_{i=1}^r (t(i) + 1)\left(\sum\limits_{i=1}^r (t(i) + 1)d_i\right)^{n-r}\right) \\ & \leq \frac{H^n}{(n-r)!}d_1d_2 \cdots d_r \left( |\{t \in F(r, q)||t| = q\}|(q/r + 1)^r \left(d_1 + d_2 + \cdots + d_{r-1} + (q + 1)d_r\right)^{n-r}\right) \\ & = \frac{H^n}{(n-r)!}d_1d_2 \cdots d_r \left( \binom{r + q - 1}{r-1}(q/r + 1)^r \left(d_1 + d_2 + \cdots + d_{r-1} + (q + 1)d_r\right)^{n-r}\right)\end{align*} 
\par Similarly, we first claim for a lower bound that $$\sum\limits_{t \in F(r, q), |t| = q} \prod\limits_{i=1}^r E(e(i), t(i)) \geq 0.6 \cdot (q+1)^{e(r)}.$$ To prove this, we use Lemma \ref{lem: more useable asymptotic}. If $e(r) \geq 3r^2 + 100,$ then we obtain the desired inequality, looking at the term where $t(r) = q$ and the rest of the $t(i)$ are zero. Otherwise, $e(r) < 3r^2 + 100.$ But then there exists some $i$ for which $e(i) \geq 3r^2 + 100,$ since $n = \sum e(i) \geq 10r^4 + 1000 > 3r^3 + 100r.$ Thus, in this second case we would have $$\sum\limits_{t \in F(r, q), |t| = q} E(e(i), t(i)) \geq 0.6 \cdot (q+1)^{e(i)} > 0.6 \cdot (q+1)^{e(r)}.$$
We use this to then show that \begin{align*}|\chi(X, \Omega^q_{X/K})| & = \frac{H^n}{n!} \sum\limits_{e \in F(r, n), |e| = n} \left(\sum\limits_{t \in F(r, q), |t| = q} \prod\limits_{i=1}^r E(e(i), t(i))\right)\binom{n}{e(1), e(2), \ldots, e(r)}d_1^{e(1)} \cdots d_r^{e(r)} \\ & = H^n \sum\limits_{e \in F(r, n), |e| = n} \left(\sum\limits_{t \in F(r, q), |t| = q} \prod\limits_{i=1}^r E(e(i), t(i))\right)\prod\limits_{i=1}^r \frac{d_i^{e(i)}}{e(i)!} \\ & = H^n \sum\limits_{e \in F(r, n), |e| = n} \left(\sum\limits_{t \in F(r, q), |t| = q} \prod\limits_{i=1}^r E(e(i), t(i))\right)\prod\limits_{i=1}^r \frac{d_i}{e(i) }\frac{d_i^{e(i) - 1}}{(e(i) - 1)!} \\ & \geq H^n \frac{r^r}{n^r} \sum\limits_{e \in F(r, n), |e| = n} \left( 0.6 \cdot (q + 1)^{e(r)}\right)\prod\limits_{i=1}^r d_i \frac{d_i^{e(i) - 1}}{(e(i) - 1)!} \\ & = \frac{0.6(q+1)r^r}{n^r}\frac{H^n}{(n-r)!}d_1d_2 \cdots d_r \left(\left(d_1 + d_2 + \cdots + d_{r-1} + (q+1)d_r\right)^{n-r}\right).\end{align*} 
\end{proof}
\par So similarly to the argument in the case where all of the $d_i$ are equal to each other, we show that each term on the right-hand side is at most a fifth of the left-hand side. That is, letting $D = d_1 + d_2 + \cdots + d_r,$ we want to show these inequalities:
\begin{enumerate}
    \item $0.2 \cdot \frac{0.6\cdot 3r^r}{n^r} (D + 2d_r)^{n-r} > \frac{n-r}{2}\binom{r}{r-1}(1/r + 1)^r (D + d_r)^{n-r},$
    \item $0.2 \cdot \frac{0.6 \cdot 3r^r}{n^r} (D+2d_r)^{n-r} \cdot \sqrt{\frac{0.6 \cdot 2r^r}{n^r} (D+d_r)^{n-r}} > D^{n-r} \sqrt{18 \binom{r+4}{r-1}(5/r + 1)^r (D + 5d_r)^{n-r}},$
    \item $0.2 \cdot \frac{0.6 \cdot 3r^r}{n^r} (D+2d_r)^{n-r} > \sqrt{\frac{9}{2}\binom{r+2}{r-1}(3/r+1)^r (D+3d_r)^{n-r} \binom{r}{r-1}(1/r + 1)^r(D + d_r)^{n-r}},$
    \item $0.2 \cdot 2 \cdot 3 (\frac{0.6 r^r}{n^r})^2 (D + d_r)^{n-r}(D + 2d_r)^{n-r} > 9\binom{r + 2}{r-1}(3/r + 1)^r D^{n-r}(D + 3d_r)^{n-r},$
    \item $0.2 \cdot \frac{0.6\cdot 3r^r}{n^r} (D + 2d_r)^{n-r} > 3D^{n-r},$
\end{enumerate}
Noting that $D \leq rd_r,$ and using the inequality $\frac{r+i - 1}{i} \leq r$ for $i \geq 1,$ to prove the above inequalities it suffices to prove the following inequalities:
\begin{enumerate}
    \item $0.72 (1 + \frac{1}{r+1})^{n-r} >  \frac{n^{r+1}}{r^{r-1}} (1/r + 1)^r,$
    \item $0.00864 (1 + \frac{8r + 4}{r^3 + 5r^2})^{n-r} > \frac{n^{3r}}{r^{3r-5}} (5/r + 1)^r,$
    \item $0.0288 (1 + \frac{1}{r^2 + 4r + 3})^{(n-r)} > \frac{n^{2r}}{r^{2r-4}}(3/r+1)^r(1/r + 1)^r,$
    \item $0.048 (1 + \frac{2}{r^2 + 3r})^{n-r} > \frac{n^{2r}}{r^{2r-3}}(3/r + 1)^r,$
    \item $0.12 (1 + \frac{2}{r})^{n-r} > \frac{n^r}{r^r}.$
\end{enumerate}
Now, observe that, again, $(1 + \frac{1}{x})^x \leq e$ for positive $x,$ and for $x \geq 1,$ we know that $(1 + \frac{1}{x})^x \geq 2.$ Furthermore, for $x \geq 10,$ this is at least $2.5.$ So to prove these inequalities, it suffices to prove these five inequalities:
\begin{enumerate}
    \item $0.72 \cdot 2^{\frac{n-r}{r+1}} >  e \frac{n^{r+1}}{r^{r-1}},$
    \item $0.00864 \cdot 2^{\frac{(n-r)(8r + 4)}{r^3 + 5r^2}} > \frac{n^{3r}}{r^{3r-5}} e^5,$
    \item $0.0288 \cdot 2.5^{\frac{(n-r)}{r^2 + 4r + 3}} > \frac{n^{2r}}{r^{2r-4}}e^4,$
    \item $0.048 \cdot 2.5^{\frac{2(n-r)}{r^2 + 3r}} > \frac{n^{2r}}{r^{2r-3}}e^3,$
    \item $0.12 \cdot 2^{\frac{2(n-r)}{r}} > \frac{n^r}{r^r}.$
\end{enumerate}
Finally, we note that $\frac{8r + 4}{r^3 + 5r^2} \geq \frac{2}{r^2}, \frac{1}{r^2 + 3r} \geq \frac{1}{4r^2},$ and $\frac{1}{r^2 + 4r + 3} \geq \frac{1}{2r^2}$ if $r \geq 5.$ Furthermore, observe that for $r \geq 2,$ we have that \[\frac{r}{r+1}, \frac{r}{r^2 + 3r}, \frac{r}{r^2 + 4r + 3} \leq 1, \frac{r(8r + 4)}{r^3 + 5r^2} \leq 2.\] Thus, it suffices to prove these inequalites:
\begin{enumerate}
    \item $2^{\frac{n}{r+1}} > 10 \frac{n^{r+1}}{r^{r-1}},$
    \item $2^{\frac{2n}{r^2}} > 72000\frac{n^{3r}}{r^{3r-5}},$
    \item $2.5^{\frac{n}{2r^2}} >4015\frac{n^{2r}}{r^{2r-4}}$ if $r \geq 5,$ $2.5^{\frac{n}{r^2 + 4r + 3}} >4015\frac{n^{2r}}{r^{2r-4}}$ otherwise.
    \item $2.5^{\frac{n}{2r^2}} > 900\frac{n^{2r}}{r^{2r-3}},$
    \item $2^{\frac{2n}{r}} > 17\frac{n^r}{r^r}.$
\end{enumerate}
We observe that the third inequality implies the other inequalities, in the assumption that $\frac{n}{r} \geq 1$ and $r \geq 2.$ Let $x = \frac{n}{r^2};$ we then want to show that $$2.5^{x/2} \geq 4015 x^{2r} r^{2r + 4}$$ for $x \geq 10r^2 + \frac{1000}{r^2}.$ Taking $r$th roots, and letting $y = \frac{x}{r},$ we thus want to show that $2.5^{y/2} \geq 4015^{1/r} y^2 r^{4 + \frac{4}{r}}$ for $y \geq 10r + \frac{1000}{r^3}.$ First, suppose $r \geq 5,$ so then $4015^{1/r} \leq 5.5,$ and $4 + \frac{4}{r} \leq 5.$ Then, we want to show that $2.5^{y/2} \geq 5.5y^2r^5$ for $y \geq 10r + \frac{1000}{r^3}.$ Note first that, when $y = 10r,$ the inequality we want to prove reduces to $2.5^{5r} \geq 550 r^7,$ which is true for $r \geq 5.$ Furthermore, we can check that this inequality holding for $y = 10r$ means it holds for larger $y,$ by checking with the first and second derivatives of the two sides.
\par Now, say $r \leq 4.$ Running the same reduction as above for the inequality $2.5^{\frac{n}{r^2 + 4r + 3}} > 4015 \frac{n^{2r}}{r^{2r - 4}},$ we find that we want to show these inequalities:
\begin{enumerate}
    \item When $r = 2,$ $2.5^{4y/15} > 64 \cdot y^22^6,$
    \item When $r = 3,$ $2.5^{3y/8} > 20 \cdot y^23^6,$
    \item When $r = 4,$ $2.5^{16y/35} > 8 \cdot y^24^5,$
\end{enumerate}
These hold when $y \geq 120, y \geq 60,$ and $y \geq 40,$ respectively, which we noted is enough to prove the lemma.
\printbibliography

@article {lv18,
    AUTHOR = {Lawrence, Brian and Venkatesh, Akshay},
     TITLE = {Diophantine problems and {$p$}-adic period mappings},
   JOURNAL = {Invent. Math.},
  FJOURNAL = {Inventiones Mathematicae},
    VOLUME = {221},
      YEAR = {2020},
    NUMBER = {3},
     PAGES = {893--999},
      ISSN = {0020-9910,1432-1297},
   MRCLASS = {11G35 (11F80 11J89 14G05)},
  MRNUMBER = {4132959},
MRREVIEWER = {Carlos\ de Vera-Piquero},
       DOI = {10.1007/s00222-020-00966-7},
       URL = {https://doi.org/10.1007/s00222-020-00966-7},
}

@misc{km24,
      title={Arithmetic finiteness of very irregular varieties}, 
      author={Thomas Krämer and Marco Maculan},
      year={2024},
      eprint={2310.08485},
      archivePrefix={arXiv},
      primaryClass={math.AG},
      url={https://arxiv.org/abs/2310.08485}, 
}

@misc{jklm23,
      title={The monodromy of families of subvarieties on abelian varieties}, 
      author={Ariyan Javanpeykar and Thomas Krämer and Christian Lehn and Marco Maculan},
      year={2023},
      eprint={2210.05166},
      archivePrefix={arXiv},
      primaryClass={math.AG},
      url={https://arxiv.org/abs/2210.05166}, 
}

@misc{ls25,
      title={The Shafarevich conjecture for hypersurfaces in abelian varieties}, 
      author={Brian Lawrence and Will Sawin},
      year={2025},
      eprint={2004.09046},
      archivePrefix={arXiv},
      primaryClass={math.NT},
      url={https://arxiv.org/abs/2004.09046}, 
}

@misc{vakil,
    title = {The Rising Sea: Foundations of Algebraic Geometry},
    author = {Vakil, Ravi},
    year = {2024},
    url = {https://math.stanford.edu/~vakil/216blog/FOAGjul2724public.pdf}}

@misc{stacks-project,
  author       = {The {Stacks project authors}},
  title        = {The Stacks project},
  howpublished = {\url{https://stacks.math.columbia.edu}},
  year         = {2025},
}

@book {BLR,
    AUTHOR = {Bosch, Siegfried and L\"utkebohmert, Werner and Raynaud,
              Michel},
     TITLE = {N\'eron models},
    SERIES = {Ergebnisse der Mathematik und ihrer Grenzgebiete (3) [Results
              in Mathematics and Related Areas (3)]},
    VOLUME = {21},
 PUBLISHER = {Springer-Verlag, Berlin},
      YEAR = {1990},
     PAGES = {x+325},
      ISBN = {3-540-50587-3},
   MRCLASS = {14K15 (11G10 14L15)},
  MRNUMBER = {1045822},
MRREVIEWER = {James\ Milne},
       DOI = {10.1007/978-3-642-51438-8},
       URL = {https://doi.org/10.1007/978-3-642-51438-8},
}

@book {mumford,
    AUTHOR = {Mumford, David},
     TITLE = {Abelian varieties},
    SERIES = {Tata Institute of Fundamental Research Studies in Mathematics},
    VOLUME = {5},
      NOTE = {With appendices by C. P. Ramanujam and Yuri Manin,
              Corrected reprint of the second (1974) edition},
 PUBLISHER = {Tata Institute of Fundamental Research, Bombay; by Hindustan
              Book Agency, New Delhi},
      YEAR = {2008},
     PAGES = {xii+263},
      ISBN = {978-81-85931-86-9; 81-85931-86-0},
   MRCLASS = {14Kxx},
  MRNUMBER = {2514037},
}

@book {eulerian,
    AUTHOR = {Petersen, T. Kyle},
     TITLE = {Eulerian Numbers},
    SERIES = {Birkh\"{a}user Advanced Texts},
 PUBLISHER = {Birkh\"{a}user New York, NY},
      YEAR = {2015},
     PAGES = {xviii+456},
      ISBN = {978-1-4939-3091-3}
}

@book {lazarsfeld,
    AUTHOR = {Lazarsfeld, Robert},
     TITLE = {Positivity in algebraic geometry. {I}},
    SERIES = {Ergebnisse der Mathematik und ihrer Grenzgebiete. 3. Folge. A
              Series of Modern Surveys in Mathematics [Results in
              Mathematics and Related Areas. 3rd Series. A Series of Modern
              Surveys in Mathematics]},
    VOLUME = {48},
      NOTE = {Classical setting: line bundles and linear series},
 PUBLISHER = {Springer-Verlag, Berlin},
      YEAR = {2004},
     PAGES = {xviii+387},
      ISBN = {3-540-22533-1},
   MRCLASS = {14-02 (14C20)},
  MRNUMBER = {2095471},
MRREVIEWER = {Mihnea\ Popa},
       DOI = {10.1007/978-3-642-18808-4},
       URL = {https://doi.org/10.1007/978-3-642-18808-4},
}

@book{Green_2013, place={Cambridge}, series={Cambridge Tracts in Mathematics}, title={Combinatorics of Minuscule Representations}, publisher={Cambridge University Press}, author={Green, R. M.}, year={2013}, collection={Cambridge Tracts in Mathematics}}

@article{kram2015,
   title={Vanishing theorems for constructible sheaves on abelian varieties},
   volume={24},
   ISSN={1534-7486},
   url={http://dx.doi.org/10.1090/jag/645},
   DOI={10.1090/jag/645},
   number={3},
   journal={Journal of Algebraic Geometry},
   publisher={American Mathematical Society (AMS)},
   author={Krämer, Thomas and Weissauer, Rainer},
   year={2015},
   month=apr, pages={531–568} }

@book {fga,
    AUTHOR = {Grothendieck, Alexander},
     TITLE = {Fondements de la g\'eom\'etrie alg\'ebrique. [{E}xtraits du
              {S}\'eminaire {B}ourbaki, 1957--1962.]},
 PUBLISHER = {Secr\'etariat math\'ematique, Paris},
      YEAR = {1962},
     PAGES = {ii+205},
   MRCLASS = {14.00},
  MRNUMBER = {146040},
}

@incollection {faltings,
    AUTHOR = {Faltings, Gerd},
     TITLE = {Finiteness theorems for abelian varieties over number fields},
 BOOKTITLE = {Arithmetic geometry ({S}torrs, {C}onn., 1984)},
     PAGES = {9--27},
      NOTE = {Translated from the German original [Invent.\ Math.\ {\bf 73}
              (1983), no.\ 3, 349--366; MR0718935; ibid.\ {\bf 75} (1984),
              no.\ 2, 381; MR0732554] by Edward Shipz},
 PUBLISHER = {Springer, New York},
      YEAR = {1986},
      ISBN = {0-387-96311-1},
   MRCLASS = {11G10 (14K15)},
  MRNUMBER = {861971},
}

@article {andreK3,
    AUTHOR = {Andr\'e, Yves},
     TITLE = {On the {S}hafarevich and {T}ate conjectures for
              hyper-{K}\"ahler varieties},
   JOURNAL = {Math. Ann.},
  FJOURNAL = {Mathematische Annalen},
    VOLUME = {305},
      YEAR = {1996},
    NUMBER = {2},
     PAGES = {205--248},
      ISSN = {0025-5831,1432-1807},
   MRCLASS = {14C30 (14J10 14K99)},
  MRNUMBER = {1391213},
MRREVIEWER = {Claire\ Voisin},
       DOI = {10.1007/BF01444219},
       URL = {https://doi.org/10.1007/BF01444219},
}

@article {JavanLoughranLowHodge,
    AUTHOR = {Javanpeykar, A. and Loughran, D.},
     TITLE = {Complete intersections: moduli, {T}orelli, and good reduction},
   JOURNAL = {Math. Ann.},
  FJOURNAL = {Mathematische Annalen},
    VOLUME = {368},
      YEAR = {2017},
    NUMBER = {3-4},
     PAGES = {1191--1225},
      ISSN = {0025-5831,1432-1807},
   MRCLASS = {11G35 (14C30 14C34 14D23 14G25 14K30 14M10)},
  MRNUMBER = {3673652},
MRREVIEWER = {J\"org\ Jahnel},
       DOI = {10.1007/s00208-016-1455-5},
       URL = {https://doi.org/10.1007/s00208-016-1455-5},
}

@article {JLflag,
    AUTHOR = {Javanpeykar, A. and Loughran, D.},
     TITLE = {Good reduction of algebraic groups and flag varieties},
   JOURNAL = {Arch. Math. (Basel)},
  FJOURNAL = {Archiv der Mathematik},
    VOLUME = {104},
      YEAR = {2015},
    NUMBER = {2},
     PAGES = {133--143},
      ISSN = {0003-889X,1420-8938},
   MRCLASS = {14L15 (11E72 14G25 14M15)},
  MRNUMBER = {3306042},
MRREVIEWER = {Alan\ Koch},
       DOI = {10.1007/s00013-015-0728-7},
       URL = {https://doi.org/10.1007/s00013-015-0728-7},
}

@incollection {beauville,
    AUTHOR = {Beauville, Arnaud},
     TITLE = {Le groupe de monodromie des familles universelles
              d'hypersurfaces et d'intersections compl\`etes},
 BOOKTITLE = {Complex analysis and algebraic geometry ({G}\"ottingen, 1985)},
    SERIES = {Lecture Notes in Math.},
    VOLUME = {1194},
     PAGES = {8--18},
 PUBLISHER = {Springer, Berlin},
      YEAR = {1986},
      ISBN = {3-540-16490-1},
   MRCLASS = {14J10 (14D05 14H10 14M10)},
  MRNUMBER = {855873},
MRREVIEWER = {I.\ Dolgachev},
       DOI = {10.1007/BFb0076991},
       URL = {https://doi.org/10.1007/BFb0076991},
}
\end{document}